\documentclass[11pt]{article}
\usepackage{amssymb,amsmath,fullpage}
\usepackage{graphicx}
\numberwithin{equation}{section}
\newtheorem {thm}{Theorem}[section]
\newtheorem {lem}[thm]{Lemma}
\newtheorem {cor}[thm]{Corollary}
\newtheorem {prop}[thm]{Proposition}

\newcommand{\la}{\langle}
\newcommand{\ra}{\rangle}
\newcommand{\proof}{\par\noindent{\em Proof.\ }}

\newcommand{\qed}{\hfill $\square$\par\smallskip}

\title{\bf\boldmath 
$q$-Selberg extensions of Gasper's \\ two stage $q$-beta integrals
}
\author{
{\sc Masahiko Ito}\thanks{
Department of Mathematical Sciences, University of the Ryukyus, 
Okinawa 903-0213, Japan 
}
\ and
{\sc Peter J.~Forrester}\thanks{
School of Mathematics and Statistics, The University of Melbourne, Victoria 3010, Australia
}
}
\date{\em{\normalsize To the memory of Prof.~Masatoshi Noumi}
}

\allowdisplaybreaks 

\begin{document}
\maketitle
%
%
%
\begin{abstract}
Gasper gave a two-step extension of the Askey--Roy formula for a beta-integral type defined on the unit circle.
This involved increasing the number of parameters and adding a balancing condition. 
We derive a multidimensional generalization of the final formula in Gasper's extension. 
By considering a two-step degeneration involving parameters our generalization becomes Tarasov--Varchenko's formula --- itself a multidimensional generalization of the original Askey--Roy formula. 
The proof is done by Aomoto's method.
This generalized integration by parts strategy makes use of a class of functions known as 
the fundamental invariants of type $BC$.
\end{abstract}
%
\section{Introduction}
\label{section:1}
Throughout this paper, we fix $q \in \mathbb{C}^*$ with $|q|<1$ and use the symbols 
$(u;q)_\infty:=\prod_{i=0}^\infty(1-q^i u)$ and $(u;q)_N:=(u;q)_\infty/(q^Nu;q)_\infty$. 
We also use the abbreviation $(u_1,\ldots,u_m;q)_\infty=(u_1;q)_\infty\cdots (u_m;q)_\infty$.
We define $\theta(u;q)$ by $\theta(u;q)=(u,q/u;q)_\infty$, 
which satisfies 
\begin{equation}
\label{eq:00quasi-period}
\theta(q/u;q)=\theta(u;q)\quad\mbox{and}\quad\theta(qu;q)=-\theta(u;q)/u.
\end{equation}
We also use the abbreviation $\theta(u_1,\ldots,u_m;q)=\theta(u_1;q)\cdots \theta(u_m;q)$.
\par
In 1989 
Gasper \cite{Ga} gave a two-step extension of the Askey--Roy formula \cite{AR}
for a beta-integral type defined on the unit circle, by increasing the number of parameters and adding a balancing condition. 
The aim of this paper is to provide a multidimensional generalization 
of the final evaluation obtained after the two-step extension, which we now state. 
\begin{thm}
\label{thm:KBC}
Suppose that $|a_{i}|<1$ $(i=1,2,\ldots,6)$ and $|t|<1$.
Under the balancing condition $q=a_{1}a_{2}a_{3}a_{4}a_{5}a_{6}t^{2n-2}$, 
for $v\in \mathbb{C}^*$ we have 
\begin{equation}
\label{eq:KBC}
\begin{split}
&\frac{1}{n!}\int_{\mathbb{T}^{n}}
\prod_{i=1}^{n}
\bigg[
\frac{(z_{i}^{2};q)_{\infty}}
{(qz_{i}^{2};q)_{\infty}}
\frac{\theta(t^{n-1}a_{1}a_{2}vz_{i}^{-1}, vz_{i};q)}{\prod_{j=1}^{2}(a_jz_{i},a_jz_{i}^{-1};q)_{\infty}}
\prod_{k=3}^{6}\frac{(qa_{k}^{-1}z_{i};q)_{\infty}}{(a_{k}z_{i};q)_{\infty}}
\bigg]
\\&\qquad
\times\prod_{1\le j< k\le n}
\frac{(z_{j}z_{k};q)_{\infty}}{(qz_{j}z_{k};q)_{\infty}}
\frac{(z_{j}/z_{k},z_{k}/z_{j};q)_{\infty}}
{(tz_{j}/z_{k},tz_{k}/z_{j};q)_{\infty}}
\frac{(qt^{-1}z_{j}z_{k};q)_{\infty}}
{(tz_{j}z_{k};q)_{\infty}}
\,\varpi(z)
\\[5pt]
&\quad =
\prod_{i=1}^{n}\frac{(t;q)_{\infty}\prod_{l=1}^{2}\theta(t^{i-1}a_{l}v;q)\prod_{3\le j<k\le 6}\theta(t^{i-1}a_{j}a_{k};q)}
{(q,t^{i};q)_{\infty}\prod_{1\le j<k\le 6}(t^{i-1}a_{j}a_{k};q)_{\infty}}.
\end{split}
\end{equation}
Here the symbol $\varpi(z)$ is the measure given by 
\begin{equation}
\label{eq:varpi(z)}
\varpi(z)=
\frac{1}{(2\pi\sqrt{-1})^n}
\frac{dz_1}{z_1}\cdots\frac{dz_n}{z_n},
\end{equation}
and $\mathbb{T}^n$ is the $n$-fold direct product of the unit circle traversed in the positive direction, i.e.,  
\begin{equation}
\label{eq:n-torus}
\mathbb{T}^n=\{(z_1,\ldots,z_n)\in \mathbb{C}^n\,|\, |z_i|=1\ (i=1,\ldots,n)\}.
\end{equation}
\end{thm}

A transformation of parameters and deformation of the cycle 
--- the details of which are presented in Appendix \ref{appendix:A}
--- shows a further two multidimensional integral evaluation formulas, both of which
 are equivalent to \eqref{eq:KBC}.  
Those two formulas are used later when discussing the case where \eqref{eq:KBC} is degenerate. 
Moreover, 
we will see that
the $n=1$ case of \eqref{eq:KBC} is equivalent to the final of Gasper's extended Askey--Roy formulas \cite{Ga}. 
\par
For the first equivalent form, the required reset of the parameters in \eqref{eq:KBC} as 
$$a_{1}\to a_{1}\varepsilon,\  
a_{2}\to a_{2}\varepsilon,\ 
a_{3}\to a_{3}\varepsilon,\ 
a_{4}\to b_{2}\varepsilon^{-1},\ 
a_{5}\to b_{3}\varepsilon^{-1},\ 
a_{6}\to b_{4}\varepsilon^{-1}\
\mbox{and}\ 
v\to v \varepsilon^{-1}.
$$
This procedure is to be accompanied by 
a simultaneous transformation of the integral variables (deformation of the cycle) as 
$z_{i}\to z_{i}\varepsilon$ $(i=1,\ldots, n)$, allowing for a restatement of   \eqref{eq:KBC}.
\begin{cor} Suppose that $|a_i|<1$ $(i=1,2,3)$, $|b_j|<1$ $(j=1,2,3)$ and $|t|<1$. 
\label{coro:main-2}
Under the balancing condition 
$q=a_{1}a_{2}a_{3}b_{2}b_{3}b_{4}t^{2n-2}$, 
for $v\in \mathbb{C}^*$ we have 
\begin{equation}
\label{eq:KBCe-2}
\begin{split}
&\frac{1}{n!}\int_{\mathbb{T}^{n}}
\prod_{i=1}^{n}
\bigg[
\frac{(\varepsilon^{2}z_{i}^{2};q)_{\infty}\,
\theta(t^{n-1}a_{1}a_{2}vz_{i}^{-1}, vz_{i},a_{3}z_{i}^{-1};q)}
{(q\varepsilon^{2}z_{i}^{2};q)_{\infty}\prod_{j=1}^{3}(a_j\varepsilon^{2} z_{i},a_jz_{i}^{-1};q)_{\infty}}
\prod_{k=1}^{3}\frac{(qb_{k}^{-1}\varepsilon^{2} z_{i};q)_{\infty}}{(b_{k}z_{i};q)_{\infty}}
\bigg]
\\&\qquad\quad
\times\prod_{1\le j< k\le n}
\frac{(\varepsilon^{2} z_{j}z_{k};q)_{\infty}}{(q\varepsilon^{2} z_{j}z_{k};q)_{\infty}}
\frac{(z_{j}/z_{k},z_{k}/z_{j};q)_{\infty}}
{(tz_{j}/z_{k},tz_{k}/z_{j};q)_{\infty}}
\frac{(qt^{-1}\varepsilon^{2} z_{j}z_{k};q)_{\infty}}
{(t\varepsilon^{2} z_{j}z_{k};q)_{\infty}}
\,\varpi(z)
\\
&\quad =
\prod_{i=1}^{n}\frac{(t;q)_{\infty}\prod_{l=1}^{2}\theta(t^{i-1}a_{l}v;q)
\prod_{k=1}^3\theta(t^{i-1}a_3b_k;q)
\prod_{1\le j<k\le 3}(qt^{-(i-1)}b_{j}^{-1}b_{k}^{-1}\varepsilon^{2};q)}
{(q,t^{i};q)_{\infty}
\prod_{1\le j<k\le 3}(t^{i-1}a_{j}a_{k}\varepsilon^{2};q)_{\infty}
\prod_{j,k=1}^{3}(t^{i-1}a_{j}b_{k};q)_{\infty}}.
\end{split}
\end{equation}
\end{cor}
\par
On the other hand, after we reset the parameters in \eqref{eq:KBC} as 
$$a_{1}\to a_{1}\varepsilon,\  
a_{2}\to a_{2}\varepsilon,\ 
a_{3}\to b_{1}\varepsilon^{-1},\ 
a_{4}\to b_{2}\varepsilon^{-1},\ 
a_{5}\to b_{3}\varepsilon^{-1},\ 
a_{6}\to b_{4}\varepsilon^{-1}\
\mbox{and}\ 
v\to v \varepsilon^{-1}, 
$$
and simultaneously deform the integration cycles 
$z_{i}\to z_{i}\varepsilon$ $(i=1,\ldots, n)$, 
a second restatement of 
\eqref{eq:KBC} results.
\begin{cor} 
Suppose that $|a_i|<1$ $(i=1,2)$, $|b_j|<1$ $(j=1,\ldots,4)$ and $|t|<1$. 
Under the balancing condition 
$q\varepsilon^{2}=a_{1}a_{2}b_{1}b_{2}b_{3}b_{4}t^{2n-2}$, 
for $v\in \mathbb{C}^*$ we have 
\begin{equation}
\label{eq:KBCe}
\begin{split}
&\frac{1}{n!}\int_{\mathbb{T}^{n}}
\prod_{i=1}^{n}
\bigg[\frac{
(\varepsilon z_{i}, -\varepsilon z_{i};q)_{\infty}
}
{
(q\varepsilon z_{i}, -q\varepsilon z_{i};q)_{\infty}
}
\frac{\theta(t^{n-1}a_{1}a_{2}vz_{i}^{-1}, vz_{i};q)}{\prod_{j=1}^{2}(a_j\varepsilon^{2} z_{i},a_jz_{i}^{-1};q)_{\infty}}
\prod_{k=1}^{4}\frac{(qb_{k}^{-1}\varepsilon^{2} z_{i};q)_{\infty}}{(b_{k}z_{i};q)_{\infty}}
\bigg]
\\&\qquad\quad
\times\prod_{1\le j< k\le n}
\frac{(\varepsilon^{2} z_{j}z_{k};q)_{\infty}}{(q\varepsilon^{2} z_{j}z_{k};q)_{\infty}}
\frac{(z_{j}/z_{k},z_{k}/z_{j};q)_{\infty}}
{(tz_{j}/z_{k},tz_{k}/z_{j};q)_{\infty}}
\frac{(qt^{-1}\varepsilon^{2} z_{j}z_{k};q)_{\infty}}
{(t\varepsilon^{2} z_{j}z_{k};q)_{\infty}}
\,\varpi(z)
\\
&\quad =
\prod_{i=1}^{n}\frac{(t;q)_{\infty}\prod_{l=1}^{2}\theta(t^{i-1}a_{l}v;q)
\prod_{1\le j<k\le 4}(qt^{-(i-1)}b_{j}^{-1}b_{k}^{-1}\varepsilon^{2};q)}
{(q,t^{i};q)_{\infty}(t^{i-1}a_{1}a_{2}\varepsilon^{2};q)_{\infty}\prod_{j=1}^{2}\prod_{k=1}^{4}(t^{i-1}a_{j}b_{k};q)_{\infty}}.
\end{split}
\end{equation}
\end{cor}
\par
If $n=1$, \eqref{eq:KBCe} coincides exactly  with 
Gasper's further extended Askey--Roy formula \cite[(3.10)]{Ga} \cite[(4.11.4)]{GR} of Gasper--Rahman's book by substitution with
$\varepsilon =\sqrt{a}$, $a_{1}=\alpha$, $a_{2}=\beta$, $b_{1}=b$, $b_{2}=c$, $b_{3}=d$, $b_{4}=f$ and $v=q/\gamma$. 
Also, 
There is further immediate consequence in the form of
\eqref{eq:KBC} as exhibited in \eqref{eq:KBCe}. 
Thus take the limit $\varepsilon \to 0$ 
for fixed parameters $a_{1}$, $a_{2}$, $b_{1}$, $b_{2}$, $b_{3}$ and $t$. 
Then $b_{4}=q\varepsilon^{2}/a_{1}a_{2}b_{1}b_{2}b_{3}t^{2n-2}\to 0$ and $qb_{4}^{-1}\varepsilon^{2}$ is fixed as $qb_{4}^{-1}\varepsilon^{2}=a_{1}a_{2}b_{1}b_{2}b_{3}t^{2n-2}$. Consequently, without requiring a balancing condition, we have
\begin{equation*}
\begin{split}
&\frac{1}{n!}\int_{\mathbb{T}^{n}}
\prod_{i=1}^{n}
\bigg[
\frac{\theta(t^{n-1}a_{1}a_{2}vz_{i}^{-1}, vz_{i};q)}{\prod_{j=1}^{2}(a_jz_{i}^{-1};q)_{\infty}}
\frac{(t^{2n-2}a_{1}a_{2}b_{1}b_{2}b_{3}\, z_{i};q)_{\infty}}{\prod_{k=1}^{3}(b_{k}z_{i};q)_{\infty}}
\bigg]
\prod_{1\le j< k\le n}
\frac{(z_{j}/z_{k},z_{k}/z_{j};q)_{\infty}}
{(tz_{j}/z_{k},tz_{k}/z_{j};q)_{\infty}}
\,\varpi(z)
\\
&\quad =
\prod_{i=1}^{n}\frac{(t;q)_{\infty}\prod_{l=1}^{2}\theta(t^{i-1}a_{l}v;q)
\prod_{j=1}^{3}(t^{2n-i-1}b_{j}^{-1}a_{1}a_{2}b_{1}b_{2}b_{3};q)}
{(q,t^{i};q)_{\infty}\prod_{j=1}^{2}\prod_{k=1}^{3}(t^{i-1}a_{j}b_{k};q)_{\infty}}.
\end{split}
\end{equation*}
Reimposing a balancing condition by setting 
$a_{3}:=q/a_{1}a_{2}b_{1}b_{2}b_{3}t^{2n-2}$,
this is the formula we discussed in \cite{FIN2026}. 
\begin{cor}
\label{coro:main}
Suppose $|a_i|<1$ $(i=1,2)$, $|b_j|<1$ $(j=1,2,3)$ and $|t|<1$.
Under the balancing condition $a_1a_2a_3b_1b_2b_3t^{2n-2}=q$, for $v\in \mathbb{C}^*$ we have 
\begin{equation}
\label{eq:main}
\begin{split}
&\frac{1}{n!}\int_{\mathbb{T}^n} \prod_{i=1}^n\frac{
\theta(t^{n-1}a_1a_2vz_i^{-1}, vz_i,a_3z_i^{-1};q)}{\prod_{m=1}^{3}(a_mz_i^{-1},b_mz_i;q)_\infty}
\prod_{1\le j<k\le n}\frac{(z_j/z_k,z_k/z_j;q)_\infty}{(tz_j/z_k,tz_k/z_j;q)_\infty}
\varpi(z)\\
&\quad=\prod_{i=1}^n
\frac{(t;q)_\infty\prod_{l=1}^3\theta(t^{i-1}a_3b_l;q)\prod_{k=1}^2\theta(t^{i-1}a_kv;q)}{(q,t^i;q)_\infty\prod_{k,l=1}^3(t^{i-1}a_kb_l;q)_\infty}.
\end{split}
\end{equation}
\end{cor}
An alternative, more direct derivation of  \eqref{eq:main} is to use \eqref{eq:KBCe-2}.
Thus  \eqref{eq:main} results from the later 
by taking the limit $\varepsilon \to 0$ 
for fixed parameters $a_{1}$, $a_{2}$, $a_{3}$, $b_{1}$, $b_{2}$, $b_{3}$ and $t$, 
then the formula in Corollary \ref{coro:main}
can be obtained more directly.
\par
If $n=1$, \eqref{eq:main} coincides with Gasper's intermediate step extended Askey--Roy formula 
\cite[(3.9)]{Ga} \cite[(4.11.3)]{GR}. 
Taking the limit $b_3\to 0$ for fixed parameters $a_{1}, a_{2}, b_{1}, b_{2}$ and $t$, 
then $qa_{3}^{-1}=a_1a_{2}b_{1}b_{2}b_{3}t^{2n-2}$ $\to 0$, we see that the formula \eqref{eq:main} degenerates to 
Tarasov--Varchenko's multidimensional generalization of the original Askey--Roy formula \cite[Example (5.14)]{TV97}; 
see \cite[Introduction]{FIN2026} for details. 
\par
In \cite{FIN2026} it was pointed out that the formula \eqref{eq:main} 
is similar in structure to 
Gustafson's multidimensional 
generalization \cite{Gu1994} of Nassrallah--Rahman integral \cite{NR1985}, 
which we recall for comparison. 
\begin{prop}{\rm\cite[Theorem 2.1]{Gu1994}}
Suppose that $|a_i|<1$ $(i=1,2,\ldots,5)$ and $|t|<1$. 
Under the balancing condition $a_1a_2\cdots a_6t^{2n-2}=q$, it follows that 
\begin{equation}
\label{eq:mNRint}
\begin{split}
&\frac{1}{2^nn!}\int_{\mathbb{T}^n}
\prod_{i=1}^n\frac{(z_i^2,z_i^{-2},qa_6^{-1}z_i,qa_6^{-1}z_i^{-1};q)_\infty}
{\prod_{m=1}^5(a_mz_i,a_mz_i^{-1};q)}
\\&\qquad\times
\prod_{1\le j<k\le n}\frac{(z_jz_k^{-1},z_j^{-1}z_k,z_jz_k,z_j^{-1}z_k^{-1};q)_\infty}
{(tz_jz_k^{-1},tz_j^{-1}z_k,tz_jz_k,tz_j^{-1}z_k^{-1};q)_\infty}
\,\varpi(z)
\\[5pt]
&\quad
=\prod_{i=1}^n\frac{(t;q)_\infty\prod_{k=1}^5(qa_6^{-1}a_k^{-1}t^{-(i-1)};q)_\infty}
{(q,t^i;q)_\infty\prod_{1\le k\le l\le 5}(a_ka_lt^{i-1};q)_\infty}.
\end{split}
\end{equation}
\end{prop}
A proof of \eqref{eq:mNRint} is given
in \cite{Ito2011, INIntBC} according to the strategy of Aomoto's method \cite{AAR99,Ao87, PJF-Book}.
A key step is the use of a class of 
functions known as fundamental invariants of type $BC$ \cite{IN2016}. 
In this paper the formula \eqref{eq:KBC} will also be proven by Aomoto's method, 
and use made of the fundamental invariants of type $BC$.
\par
Although the formulas \eqref{eq:main} and \eqref{eq:mNRint} are similar in structure, there is also an important difference. 
This is that the parameter $v$ of \eqref{eq:main} does not enter the balancing condition. 
Moreover, the integral \eqref{eq:main} is quasiperiodic in $v$, 
which shows a qualitative difference of this parameter from the other ones. 
This question was left open in our previous paper \cite{FIN2026} --- providing an explanation is also an aim of the present work. 
\par
Comparing \eqref{eq:mNRint} and \eqref{eq:KBC}, the latter being an extension of \eqref{eq:main}, 
they appear to be different. 
In fact the only difference between the integrand of \eqref{eq:mNRint} and that of \eqref{eq:KBC} is 
a pseudo-constant ($q$-periodic function), which includes the extra parameter $v$. 
As we will see later in Lemma \ref{lem:nabla varphi=0}, 
the $q$-difference equations satisfied by the integrals defined by \eqref{eq:KBC} or \eqref{eq:mNRint} depend only on their integrands. 
A consequence is that we can ignore the difference in the pseudo-constant between the two integrands and consider these two integrals to be solutions of a common $q$-difference equation. 
The only distinction between the two solutions is the boundary conditions. 
Since the parameter $v$ is included only in the pseudo-constant, 
it does not affect the form of the $q$-difference equation satisfied by the integral of \eqref{eq:KBC}. 
Therefore, it does not affect the balancing condition. 
Since the boundary conditions are determined by the location of the poles of the integrand, 
the distinction between the two solutions depends on 
the geometric location of the poles contained in the pseudo-constants.
Since the $q$-difference equation is rank 1 (i.e., a two-term recurrence), 
the connection between these two solutions can be made explicit.
\begin{prop}
Let $J$ and $I$ denote the integrals given by the left-hand side of \eqref{eq:KBC} and \eqref{eq:mNRint} respectively.
Under the balancing condition $q=a_{1}a_{2}a_{3}a_{4}a_{5}a_{6}t^{2n-2}$, we have 
\begin{equation*}
\frac{J}
{I}
=\prod_{i=1}^{n}\frac{\prod_{l=1}^{2}\theta(t^{i-1}a_{l}v;q)\prod_{3\le j<k\le 5}\theta(t^{i-1}a_{j}a_{k};q)}
{\prod_{l=1}^{2}\theta(t^{i-1}a_{l}a_{6};q)}.
\end{equation*}
\end{prop}  
\par
This paper is organized as follows. 
In Section \ref{section:2}, we describe a method for degenerating the integral 
in the main theorem involving progressively specializing the parameters. 
Important for this purpose are the $q$-difference equations with respect to 
those parameters satisfied by the integral.
Ultimately, the proof of the main theorem reduces to an integral 
when the parameters are specialized in a convenient way, 
and the proof itself becomes clear from existing results.
In Section \ref{section:3}, we will discuss the properties of 
certain interpolation functions (fundamental invariants of type $BC$) and their zeros. 
These play an important role in finding the $q$-difference equation that the integral satisfies using Aomoto's method. 
In Section \ref{section:4}, we introduce Aomoto's method. This section constitutes the technical part of this paper, 
notwithstanding that in its simplest form it corresponds to integration by parts in calculus.
In Section \ref{section:5}, based on Aomoto's method, we will prove the remaining key lemma (Lemma \ref{lem:recurrenceJ}). 
In the Appendix \ref{appendix:A}, we explain how the integral (1.5) can be derived from (1.2).
%
%
\section{Residue calculation and steps to prove Theorem \ref{thm:KBC}}
\label{section:2}
The proof of Theorem \ref{thm:KBC} is carried out by reducing to successive special cases. 
In this section we explain these degenerations. 
\par
For $x=(x_{1},\ldots,x_{n})\in (\mathbb{C}^*)^n$ and $\nu=(\nu_{1},\ldots,\nu_{n})\in \mathbb{Z}^{n}$, 
we use the symbol $xq^{\nu}$ as 
$$
xq^{\nu}:=(x_{1}q^{\nu_{1}},\ldots,x_{n}q^{\nu_{n}})\in (\mathbb{C}^*)^n. 
$$
Let $\zeta_{j}$ $(j=0,1,\ldots,n)$ be the points in $(\mathbb{C}^*)^n$ defined by 
\begin{equation}
\label{eq:zeta_j}
\zeta_j
=(\underbrace{a_{1},a_{1}t,\ldots,a_{1} t^{j-1}\phantom{\Big|}\!\!}_j, 
\underbrace{a_{2},a_{2} t,\ldots,a_{2} t^{n-j-1}\phantom{\Big|}\!\!}_{n-j})\in (\mathbb{C}^*)^n.
\end{equation}
\begin{lem} 
\label{lem:J(v)/factor}
Let $J(v)$ be the left-hand side of \eqref{eq:KBC}. 
$J(v)$ is divisible by $\prod_{i=1}^{n}\theta(t^{i-1}a_{1}v,t^{i-1}a_{2}v;q)$, and 
$J(v)/\prod_{i=1}^{n}\theta(t^{i-1}a_{1}v,t^{i-1}a_{2}v;q)$ is independent of $v$.
\end{lem}
%
\proof 
Let $\mathfrak{S}_{n}$ be the symmetric group on $\{1,2,\ldots, n\}$. 
For an arbitrary point 
$\xi=(\xi_{1},\xi_{2},\ldots, \xi_{n})\in (\mathbb{C}^*)^n$ and $\sigma\in \mathfrak{S}_{n}$, 
we define the point $\sigma\xi$ by 
$\sigma\xi=(\xi_{\sigma(1)},\xi_{\sigma(2)},\ldots, \xi_{\sigma(n)})$, and 
for $L\subset (\mathbb{C}^*)^n$ we simply write $\sigma L:=\{\sigma \xi\,|\,\xi\in L\}$. 
\par
We denote by $\tilde\Psi(z)$ the integrand of $J(v)$.  Since $\tilde\Psi(z)$ is $\mathfrak{S}_{n}$-symmetric, we have 
$$
J(v)=\frac{1}{n!}\int_{\mathbb{T}^{n}}\tilde\Psi(z)\varpi(z)=
\frac{1}{n!}\sum_{\sigma\in\mathfrak{S}_{n}}\int_{\sigma\tilde{\mathbb{T}}^{n}}\tilde\Psi(z)\varpi(z)
=
\frac{1}{n!}\sum_{\sigma\in\mathfrak{S}_{n}}\int_{\tilde{\mathbb{T}}^{n}}\sigma\tilde\Psi(z)\varpi(z)
=\int_{\tilde{\mathbb{T}}^{n}}\tilde\Psi(z)\varpi(z),
$$
where the integration cycle $\tilde{\mathbb{T}}^{n}$ is 
$$\tilde{\mathbb{T}}^{n}
=\{(z_1,\ldots,z_n)\in \mathbb{C}^n\,|\, |z_1|=1\ \mbox{and}\ |z_{i}/z_{i-1}|=1\ (i=2,\ldots,n)\}.
$$
Since the measure $\varpi(z)$ is equal to 
$$
\tilde\varpi(z)=
\frac{1}{(2\pi\sqrt{-1})^n}
\frac{dz_{1}}{z_{1}}
\frac{d(z_{2}/z_{1})}{z_{2}/z_{1}}
\frac{d(z_{3}/z_{2})}{z_{3}/z_{2}}
\cdots \frac{d(z_{n}/z_{n-1})}{z_{n}/z_{n-1}},
$$
we have 
$$
J(v)=\int_{\tilde{\mathbb{T}}^{n}}\tilde\Psi(z)\tilde\varpi(z).
$$
\par
Now we calculate the residue series in the region 
$|z_{1}|<1, |z_{2}/z_{1}|<1,\ldots, |z_{n}/z_{n-1}|<1$. 
For this, we temporarily assume $|a_{2}|<|a_{1}|$. 
We set $\tau$ as $t=q^{\tau}$. 
The integrand $\tilde\Psi(z)$ of $J(v)$ can be factored as 
$\tilde\Psi(z)=P(z;v)Q(z)$, where
\begin{align}
P(z;v)&=\prod_{i=1}^{n}
\frac{\theta(t^{n-1}a_{1}a_{2}vz_{i}^{-1},vz_{i};q)}
{z_{i}^{1+2\tau(n-i)}\theta(a_{1}z_{i}^{-1},a_2z_{i}^{-1};q)}
\prod_{1\le j< k\le n}
\frac{\theta(z_{j}/z_{k};q)}
{\theta(tz_{j}/z_{k};q)},
\label{eq:P(z;v)}\\
\begin{split}
Q(z)&=\prod_{i=1}^{n}
z_{i}^{1+2\tau(n-i)}(1-z_{i}^{2})
\prod_{k=1}^{6}
\frac{(qa_{k}^{-1}z_{i};q)_{\infty}}{(a_{k}z_{i};q)_{\infty}}
\\&\qquad
\times\prod_{1\le j< k\le n}
(1-z_{k}/z_{j})(1-z_{j}z_{k})\frac{(qt^{-1}z_{k}/z_{j},qt^{-1}z_{j}z_{k};q)_{\infty}}
{(tz_{k}/z_{j},tz_{j}z_{k};q)_{\infty}}.
\end{split}
\label{eq:Q(z)}
\end{align}
$P(z;v)$ is 
$q$-periodic with respect to each variable $z_i$. 
Taking into account the condition $|a_{1}|<1, \ldots,$ $|a_{6}|<1$ and $|t|<1$, 
the set of poles lying on the set 
$\prod_{i=1}^{n}(a_{1}z_{i}^{-1},a_{2}z_{i}^{-1};q)_{\infty}
\prod_{1\le j< k\le n}(tz_{j}/z_{k};q)_{\infty}=0$ is within the specified region. 
The set of poles can be written as 
$
\bigcup_{j=0}^{n}\{\zeta_{j}q^{\nu}\,|\,\nu\in \Lambda_{j}\}$, where 
\begin{equation*}
\Lambda_{j}:=
\bigg\{\nu=(\nu_{1},\ldots,\nu_{n})\in \mathbb{Z}^{n}\,\bigg|\,
\substack{\displaystyle 0\le \nu_{1}\le \nu_{2}\le \cdots\le \nu_{j}\\[2pt]
\displaystyle 0\le \nu_{j+1}\le \nu_{j+2}\le \cdots\le \nu_{n}}
\bigg\}.
\end{equation*}
Therefore the integral $J(v)$ is written as the sum of residues as
\begin{equation}
\label{eq:ResidueKBC}
\begin{split}
J(v)&=\int_{\tilde{\mathbb{T}}^{n}}P(z;v)Q(z)\tilde\varpi(z)
=\sum_{j=0}^{n}\sum_{\nu\in \Lambda_{j}}
\mathop{\mathrm{Res}}_{z=\zeta_jq^{\nu}}P(z;v)Q(z)\tilde\varpi(z)
\\&
=\sum_{j=0}^{n}
\mathop{\mathrm{Res}}_{z=\zeta_j}P(z;v)\tilde\varpi(z)\sum_{\nu\in \Lambda_{j}}Q(\zeta_jq^{\nu}).
\end{split}
\end{equation}
Since the factor in $P(z;v)$ that contains $v$ is $\prod_{i=1}^{n}\theta(t^{n-1}a_{1}a_{2}vz_{i}^{-1},vz_{i};q)$
in the numerator of $P(z;v)$ only, 
the factor in 
$
\mathop{\mathrm{Res}}\limits_{z=\zeta_j}P(z;v)\varpi(z)
$ that contains $v$ is $\prod_{i=1}^{n}\theta(t^{n-1}a_{1}a_{2}vz_{i}^{-1},vz_{i};q)\big|_{z=\zeta_{j}}$ only, 
and is calculated as  
\begin{align*}
&\prod_{i=1}^{n}\theta(t^{n-1}a_{1}a_{2}vz_{i}^{-1},vz_{i};q)\big|_{z=\zeta_{j}}\\
&=
\prod_{i=1}^{j}\theta(t^{n-1}a_{1}a_{2}v(a_{1}t^{i-1})^{-1},va_{1}t^{i-1};q)
\prod_{i=1}^{n-j}\theta(t^{n-1}a_{1}a_{2}v(a_{2}t^{i-1})^{-1},va_{2}t^{i-1};q)\\
&=\prod_{i=1}^{j}\theta(t^{n-i}a_{2}v,t^{i-1}a_{1}v;q)
\prod_{i=1}^{n-j}\theta(t^{n-i}a_{1}v,t^{i-1}a_{2}v;q)
=\prod_{i=1}^{n}\theta(t^{i-1}a_{1}v,t^{i-1}a_{2}v;q).
\end{align*}
Since this factor is independent of $j$ $(j=0,\ldots,n)$, $J(v)$ is divisible by $\prod_{i=1}^{n}\theta(t^{i-1}a_{1}v,t^{i-1}a_{2}v;q)$, 
and $J(v)/\prod_{i=1}^{n}\theta(t^{i-1}a_{1}v,t^{i-1}a_{2}v;q)$ is independent of $v$. 
With this established, the assumption $|a_{2}|<|a_{1}|$ can be removed, 
giving the result as stated.  
\qed
\medskip
\noindent
{\bf Remark 1.} From the residue series expression \eqref{eq:ResidueKBC}, the formula \eqref{eq:KBC} in Theorem \ref{thm:KBC} is equivalent to the identity
\begin{equation}
\label{eq:ResidueKBC2}
\sum_{j=0}^{n}
\mathop{\mathrm{Res}}_{z=\zeta_j}P(z;v)\tilde\varpi(z)\sum_{\nu\in \Lambda_{j}}Q(\zeta_jq^{\nu})
=\prod_{i=1}^{n}\frac{(t;q)_{\infty}\prod_{l=1}^{2}\theta(t^{i-1}a_{l}v;q)\prod_{3\le j<k\le 6}\theta(t^{i-1}a_{j}a_{k};q)}
{(q,t^{i};q)_{\infty}\prod_{1\le j<k\le 6}(t^{i-1}a_{j}a_{k};q)_{\infty}},
\end{equation}
under the condition $q=a_{1}a_{2}a_{3}a_{4}a_{5}a_{6}t^{2n-2}$. 
In particular, if $n=1$, 
$\displaystyle \mathop{\mathrm{Res}}_{z=\zeta_j}P(z;v)\tilde\varpi(z)$ and 
$\sum\limits_{\nu\in \Lambda_{j}}Q(\zeta_jq^{\nu})$\\[-8pt] $(j=0,1)$ 
coincide with $\displaystyle \mathop{\mathrm{Res}}_{z=a_{i}}P(z;v)\textstyle\frac{dz}{2\pi \sqrt{-1}z}$ 
and 
$\textstyle \sum\limits_{\nu=0}^{\infty}Q(a_{i}q^{\nu})$ $(i=2,1)$, respectively.
According to \eqref{eq:P(z;v)} and \eqref{eq:Q(z)} 
we have 
$\displaystyle
\mathop{\mathrm{Res}}_{z=a_{i}}P(z;v)\textstyle\frac{dz}{2\pi \sqrt{-1}z}
=\frac{(-1)^{i-1}\theta(a_{1}v,a_{2}v;q)}{(q;q)^{2}a_{1}\theta(a_{2}/a_{1};q)}
$ $(i=1,2)$
and the latter are expressed
 by the basic hypergeometric series
$$
\sum_{\nu=0}^{\infty}Q(a_i q^{\nu})=
{}_8\psi_{8}\bigg[
\begin{array}{c}
  q\xi,-q\xi,a_{1}\xi,a_{2}\xi, \ldots,a_{6}\xi\\
  \xi,-\xi,q\xi/a_{1},q\xi/a_{2},\ldots,q\xi/a_{6}
\end{array}
;q,q
\bigg]\, 
\xi(1-\xi^{2})\prod_{m=1}^{6}\frac{(q\xi/a_{m};q)_{\infty}}{(a_{m}\xi;q)_{\infty}}
\bigg|_{\xi=a_{i}}
$$
for $i=1,2$. We can confirm that \eqref{eq:ResidueKBC2} for the case $n=1$ then coincides exactly with 
Bailey's ${}_8\phi_{7}$ summation formula \cite[p.\,54 (2.11.7)]{GR} 
of Gasper--Rahman's book by substitution with
$a_{1}^{2}=a$, $a_{1}a_{2}=b$, $a_{1}a_{3}=c$, $a_{1}a_{4}=d$, $a_{1}a_{5}=e$ and  $a_{1}a_{6}=f$. 
As a consequence it can be seen that Gasper's extended Askey--Roy formula and Bailey's ${}_8\phi_{7}$ summation formula, while different in appearance, are equivalent.
\\[10pt]
\medskip
\noindent
{\bf Remark 2.} A technical point of detail passed over in the above is that in the residue calculation \eqref{eq:ResidueKBC}, 
for convergence of the series
$\sum_{j=0}^{n}\sum\limits_{\nu\in \Lambda_{j}}
\displaystyle\mathop{\mathrm{Res}}_{z=\zeta_jq^{\nu}}P(z;v)Q(z)\tilde\varpi(z)$ one requires 
\begin{equation}
\label{eq:J_epsilon}
J_{\epsilon}:=\int_{\tilde{\mathbb{T}}_{\epsilon}\hspace{-0.5pt}{\phantom{\bar{i}}\!\!\!\!}^{n}}P(z;v)Q(z)\tilde\varpi(z)\to 0 \quad (\epsilon \to 0),
\end{equation}
where 
$\tilde{\mathbb{T}}_{\epsilon}\hspace{-6.5pt}{\phantom{|}}^{n}
=\{(z_1,\ldots,z_n)\in \mathbb{C}^n\,|\, |z_1|=\epsilon\ \mbox{and}\ |z_{i}/z_{i-1}|=\epsilon\ (i=2,\ldots,n)\}
$. 
This can be confirmed as follows. 
We first take $\epsilon>0$ as $\epsilon=\tilde\epsilon q^{N}$, 
where $\tilde\epsilon$ is fixed as $0<\tilde\epsilon<1$ and $N$ is a positive integer. 
Since $P(z;v)$ given in \eqref{eq:P(z;v)} is a continuous function on the compact set 
$\tilde{\mathbb{T}}_{\tilde\epsilon}\hspace{-7pt}{\phantom{\bar{l}}}^{n}$,
$|P(z;v)|$ is bounded on $\tilde{\mathbb{T}}_{\tilde\epsilon}\hspace{-7pt}{\phantom{\bar{l}}}^{n}$,
i.e., there exists $C_{1}>0$ such that $|P(z;v)|<C_{1}$ for 
$z\in \tilde{\mathbb{T}}_{\tilde\epsilon}\hspace{-7pt}{\phantom{\bar{l}}}^{n}$. Moreover, 
since $P(z,v)$ is $q$-periodic with respect to $z$, for $\nu\in \mathbb{Z}^{n}$ we have 
$|P(zq^{\nu},v)|=|P(z,v)|<C_{1}$. 
In particular, taking $\nu=m(1,2,\ldots,n)\in \mathbb{Z}^{n}$, where $m\in \mathbb{Z}$, 
we see $|P(z,v)|<C_{1}$ for $z\in \bigcup_{m=0}^{\infty}
\tilde{\mathbb{T}}_{\tilde\epsilon q^{m}}\hspace{-18pt}{\phantom{\bar{l}}}^{n}\hspace{9pt}
$. On the other hand 
$|Q(z)/z_{1}\cdots z_{n}|$ is also bounded, i.e., 
there exists $C_{2}>0$ such that  $|Q(z)/z_{1}\cdots z_{n}|<C_{2}$, 
because by the definition \eqref{eq:Q(z)} we see 
$Q(z)\sim \prod_{i=1}^{n}z_{i}^{1+2\tau(n-i)}$ as 
$|z_{1}|\to 0$ and $|z_{i}/z_{i-1}|\to 0$ $(i=2,\ldots,n)$.
Thus, setting $C=C_{1}C_{2}$, we have $|P(z;v)Q(z)|<C|z_{1}\cdots z_{n}|$ 
for $z\in \bigcup_{m=0}^{\infty}
\tilde{\mathbb{T}}_{\tilde\epsilon q^{m}}\hspace{-18pt}{\phantom{\bar{l}}}^{n}\hspace{9pt}
$. 
Hence for $\epsilon=\tilde\epsilon q^{N}$ we obtain
\begin{align*}
&|J_{\epsilon}|
\le 
\frac{1}{|2\pi\sqrt{-1}|{\phantom{\bar{i}}\hspace{-4pt}}^{n}}
\int_{\tilde{\mathbb{T}}_{\epsilon}\hspace{-0.5pt}{\phantom{\bar{i}}\!\!\!\!}^{n}}|P(z;v)Q(z)|
\Big|\frac{dz_{1}}{z_{1}}\Big|
\Big|\frac{d(z_{2}/z_{1})}{z_{2}/z_{1}}\Big|
\cdots \Big|\frac{d(z_{n}/z_{n-1})}{z_{n}/z_{n-1}}\Big|
\\
&\quad <\frac{C}{(2\pi)^n}\int_{|z_{1}|=\epsilon}\!\!\!\!\!\!|z_{1}|^{n-1}|dz_{1}|
\int_{|z_{2}/z_{1}|=\epsilon}\!\!\!\!\!\!|z_{2}/z_{1}|^{n-2}|d(z_{2}/z_{1})|
\cdots
\int_{|z_{n}/z_{n-1}|=\epsilon}\!\!\!\!\!\!|d(z_{n}/z_{n-1})|\\
&\quad =C\epsilon^{n(n+1)/2}\to 0 \quad(\epsilon\to 0),
\end{align*}
which proves \eqref{eq:J_epsilon}.\\
\par
As is seen in Lemma \ref{lem:J(v)/factor}, $J(v)/\prod_{i=1}^{n}\theta(t^{i-1}a_{1}v,t^{i-1}a_{2}v;q)$ is independent of $v$, and so
if we can evaluate $J(v)$ when $v=a_{6}$, then we have 
$$
J(v)=J(a_{6})\prod_{i=1}^{n}\frac{\theta(t^{i-1}a_{1}v,t^{i-1}a_{2}v;q)}{\theta(t^{i-1}a_{1}a_{6},t^{i-1}a_{2}a_{6};q)}.
$$
Therefore, once the value of $J(a_{6})$ is known specifically as \eqref{eq:KBCv=a6} below, then the proof of Theorem \ref{thm:KBC} is complete.
\begin{lem}
\label{lem:KBCv=a6}
Suppose that $|a_{i}|<1$ $(i=1,\ldots, 5)$ and $|t|<1$.
Under the balancing condition $q=a_{1}a_{2}a_{3}a_{4}a_{5}a_{6}t^{2n-2}$, we have 
\begin{equation}
\label{eq:KBCv=a6}
\begin{split}
&\frac{1}{n!}\int_{\mathbb{T}^{n}}
\prod_{i=1}^{n}
\bigg[
\frac{(z_{i}^{2};q)_{\infty}}
{(qz_{i}^{2};q)_{\infty}}
\frac{\theta(t^{n-1}a_{1}a_{2}a_{6}z_{i}^{-1},a_{6}z_{i};q)}{\prod_{j=1}^{2}(a_jz_{i},a_jz_{i}^{-1};q)_{\infty}}
\prod_{k=3}^{6}\frac{(qa_{k}^{-1}z_{i};q)_{\infty}}{(a_{k}z_{i};q)_{\infty}}
\bigg]
\\&\qquad
\times\prod_{1\le j< k\le n}
\frac{(z_{j}z_{k};q)_{\infty}}{(qz_{j}z_{k};q)_{\infty}}
\frac{(z_{j}/z_{k},z_{k}/z_{j};q)_{\infty}}
{(tz_{j}/z_{k},tz_{k}/z_{j};q)_{\infty}}
\frac{(qt^{-1}z_{j}z_{k};q)_{\infty}}
{(tz_{j}z_{k};q)_{\infty}}
\,\varpi(z)
\\[5pt]
&\quad =
\prod_{i=1}^{n}\frac{(t;q)_{\infty}\prod_{l=1}^{2}\theta(t^{i-1}a_{l}a_{6};q)\prod_{3\le j<k\le 6}\theta(t^{i-1}a_{j}a_{k};q)}
{(q,t^{i};q)_{\infty}\prod_{1\le j<k\le 6}(t^{i-1}a_{j}a_{k};q)_{\infty}}.
\end{split}
\end{equation}
\end{lem}
\par 
Note that by setting $v=a_{6}$ in Theorem \ref{thm:KBC}, 
the restriction $|a_{6}|<1$ on $a_{6}$ in the latter can be removed. 
The proof of this lemma 
proceeds via the following $q$-difference equation, 
which is key in this paper.
\begin{lem} 
\label{lem:recurrenceJ}
Let $\mathcal{J}(a_{1},\ldots,a_{6})$ be the left-hand side of \eqref{eq:KBCv=a6}. 
Under the balancing condition $q=a_{1}a_{2}a_{3}a_{4}a_{5}a_{6}t^{2n-2}$, we have 
\begin{equation}
\label{eq:recurrenceJ}
\frac{\mathcal{J}(a_{1},a_{2},\ldots,a_{5},a_{6})}
{\mathcal{J}(qa_{1},a_{2},\ldots,a_{5},q^{-1}a_{6})}
=\prod_{i=1}^{n}\prod_{k=2}^{5}\frac{1-qa_{k}^{-1}a_{6}^{-1}t^{-(i-1)}}{1-a_{1}a_{k}t^{i-1}}.
\end{equation}
\end{lem}
%
\proof 
The proof of \eqref{eq:recurrenceJ} will be given 
in Section \ref{section:5}. It is based 
on Aomoto's method 
and the use of particular 
interpolation polynomials known as fundamental invariants of type $BC$. 
We remark that in \cite[Corollary 5.3]{Ito2011} a $q$-difference equation equivalent to \eqref{eq:recurrenceJ} has previously been derived in a more general setting. 
 The proof here is suited to the present situation.
\qed
\medskip
By repeated use of \eqref{eq:recurrenceJ} for the case $a_{1}\to a_{1}q^{N}$ and $a_{6}\to a_{6}q^{-N}$ $(N\to \infty)$, 
we have 
\begin{align*}
&\mathcal{J}(a_{1},a_{2},\ldots,a_{5},a_{6})
=\mathcal{J}(qa_{1},a_{2},\ldots,a_{5},q^{-1}a_{6})\prod_{i=1}^{n}
\prod_{k=2}^{5}\frac{1-qa_{k}^{-1}a_{6}^{-1}t^{-(i-1)}}{1-a_{1}a_{k}t^{i-1}}
\\&\quad
=\mathcal{J}(q^{N}a_{1},a_{2},\ldots,a_{5},q^{-N}a_{6})
\prod_{i=1}^{n}\prod_{k=2}^{5}\frac{(qa_{k}^{-1}a_{6}^{-1}t^{-(i-1)};q)_{N}}{(a_{1}a_{k}t^{i-1};q)_{N}}
\\&\quad
=\lim_{N\to \infty}\mathcal{J}(q^{N}a_{1},a_{2},\ldots,a_{5},q^{-N}a_{6})\times
\prod_{i=1}^{n}\prod_{k=2}^{5}\frac{(qa_{k}^{-1}a_{6}^{-1}t^{-(i-1)};q)_{\infty}}{(a_{1}a_{k}t^{i-1};q)_{\infty}}.
\end{align*}
Now, taking into account that $t^{n-1}a_{1}a_{2}a_{6}=q/a_{3}a_{4}a_{5}t^{n-1}$, we have 
\begin{equation}
\label{eq:KBClimit}
\begin{split}
&\lim_{N\to \infty}\mathcal{J}(q^{N}a_{1},a_{2},\ldots,a_{5},q^{-N}a_{6})\\
&\quad=
\frac{1}{n!}\int_{\mathbb{T}^{n}}
\prod_{i=1}^{n}
\bigg[
\frac{(z_{i}^{2};q)_{\infty}}
{(qz_{i}^{2};q)_{\infty}}
\frac{\theta(t^{n-1}a_{3}a_{4}a_{5}z_{i};q)}{(a_2z_{i},a_2z_{i}^{-1};q)_{\infty}}
\prod_{k=3}^{5}\frac{(qa_{k}^{-1}z_{i};q)_{\infty}}{(a_{k}z_{i};q)_{\infty}}
\bigg]
\\&\qquad\quad
\times\prod_{1\le j< k\le n}
\frac{(z_{j}z_{k};q)_{\infty}}{(qz_{j}z_{k};q)_{\infty}}
\frac{(z_{j}/z_{k},z_{k}/z_{j};q)_{\infty}}
{(tz_{j}/z_{k},tz_{k}/z_{j};q)_{\infty}}
\frac{(qt^{-1}z_{j}z_{k};q)_{\infty}}
{(tz_{j}z_{k};q)_{\infty}}
\,\varpi(z). 
\end{split}
\end{equation}
Therefore, once the value of integral \eqref{eq:KBClimit} is specified, 
the proof of  \eqref{eq:KBCv=a6} in Lemma \ref{lem:KBCv=a6} is then finalized, and so too is 
the proof of Theorem \ref{thm:KBC}.
\begin{lem}
Suppose that $|a_{i}|<1$ $(i=2,3,4,5)$ and $|t|<1$. Then we have 
\begin{equation}
\label{eq:MacSum}
\begin{split}
&
\frac{1}{n!}\int_{\mathbb{T}^{n}}
\prod_{i=1}^{n}
\bigg[
\frac{(z_{i}^{2};q)_{\infty}}
{(qz_{i}^{2};q)_{\infty}}
\frac{\theta(t^{n-1}a_{3}a_{4}a_{5}z_{i};q)}{(a_2z_{i},a_2z_{i}^{-1};q)_{\infty}}
\prod_{k=3}^{5}\frac{(qa_{k}^{-1}z_{i};q)_{\infty}}{(a_{k}z_{i};q)_{\infty}}
\bigg]
\\&\qquad\quad
\times\prod_{1\le j< k\le n}
\frac{(z_{j}z_{k};q)_{\infty}}{(qz_{j}z_{k};q)_{\infty}}
\frac{(z_{j}/z_{k},z_{k}/z_{j};q)_{\infty}}
{(tz_{j}/z_{k},tz_{k}/z_{j};q)_{\infty}}
\frac{(qt^{-1}z_{j}z_{k};q)_{\infty}}
{(tz_{j}z_{k};q)_{\infty}}
\,\varpi(z)\\[5pt]
&\quad =
\prod_{i=1}^{n}\frac{(t, t^{2n-i-1}a_{2}a_{3}a_{4}a_{5};q)_{\infty}
\prod_{3\le j<k\le 5}(qa_{j}^{-1}a_{k}^{-1}t^{-(i-1)};q)_{\infty}}
{(q,t^{i};q)_{\infty}\prod_{l=3}^{5}(t^{i-1}a_{2}a_{l};q)_{\infty}}.
\end{split}
\end{equation}
\end{lem}
%
\proof 
%
The integrand of the left-hand side of \eqref{eq:KBC} is rewritten as 
$P(z)M(z)$, where
\begin{align*}
P(z)&=\prod_{i=1}^{n}
\frac{\theta(z_{i}^{2},t^{n-1}a_{3}a_{4}a_{5}z_{i};q)}{\theta(a_2z_{i}^{-1};q)\prod_{k=2}^{5}\theta(a_{k}z_{i};q)}
\prod_{1\le j< k\le n}
\frac{\theta(z_{j}z_{k};q)}{\theta(tz_{j}z_{k};q)}
\frac{\theta(z_{j}/z_{k},z_{k}/z_{j};q)}
{\theta(tz_{j}/z_{k},tz_{k}/z_{j};q)},
\\
M(z)&=\prod_{i=1}^{n}
\frac{\prod_{k=2}^{5}(qa_{k}^{-1}z_{i},qa_{k}^{-1}z_{i}^{-1};q)_{\infty}}{(qz_{i}^{2},qz_{i}^{-2};q)_{\infty}}
\\&\qquad
\times\prod_{1\le j< k\le n}
\frac{(qt^{-1}z_{j}z_{k}^{-1},qt^{-1}z_{j}^{-1}z_{k},qt^{-1}z_{j}z_{k},qt^{-1}z_{j}^{-1}z_{k}^{-1};q)_{\infty}}
{(qz_{j}z_{k}^{-1},qz_{j}^{-1}z_{k},qz_{j}z_{k},qz_{j}^{-1}z_{k}^{-1};q)_{\infty}}.
\nonumber
\end{align*}
The function $P(z)$ is $\mathfrak{S}_{n}$-symmetric and $q$-periodic with respect to $z$. 
The function $M(z)$ is %
also $\mathfrak{S}_{n}$-symmetric.  
For $\mathbb{T}^{n}$, we calculate the residue series in the region 
$|z_{1}|<1, |z_{2}|<1,\ldots, |z_{n}|<1$. 
Then, taking account of the condition $|a_{2}|<1,\ldots,|a_{5}|<1$ and $|t|<1$, 
the set of poles lying on the set 
$\prod_{i=1}^{n}
(a_{2}z_{i}^{-1};q)_{\infty}\prod_{1\le j< k\le n}(tz_{j}/z_{k},tz_{k}/z_{j};q)_{\infty}=0$ 
is within the specified region and is written as 
$
\bigcup_{\sigma\in \mathfrak{S}_{n}}\sigma\{\zeta q^{\nu}\,|\,\nu\in \Lambda\}$, where 
$
\zeta=(a_{2},a_{2} t,\ldots,a_{2} t^{n-1})\in (\mathbb{C}^*)^n
$ and $\Lambda=\{\nu\in \mathbb{Z}^{n}\,|\, 0\le \nu_{1}\le \cdots\le \nu_{n}\}$. 
Therefore the integral of \eqref{eq:KBC} is written as the sum of residues as
\begin{align}
&\frac{1}{n!}\int_{\mathbb{T}^{n}}P(z)M(z)\varpi(z)
=\frac{1}{n!}\sum_{\sigma\in \mathfrak{S}_{n}}\sum_{\nu\in \Lambda}
\mathop{\mathrm{Res}}_{z=\sigma(\zeta q^{\nu})}P(z)M(z)\varpi(z)
\nonumber\\&\quad=\frac{1}{n!}\sum_{\sigma\in \mathfrak{S}_{n}}\sum_{\nu\in \Lambda}
\mathop{\mathrm{Res}}_{z=\zeta q^{\nu}}\sigma P(z)\sigma M(z)\varpi(z)
=\sum_{\nu\in \Lambda}
\mathop{\mathrm{Res}}_{z=\zeta q^{\nu}}P(z)M(z)\varpi(z)
\nonumber\\
&\quad=
\mathop{\mathrm{Res}}_{z=\zeta}P(z)\varpi(z) \sum_{\nu\in \Lambda}M(\zeta q^{\nu}). 
\label{eq:ResidueMacSum}
\end{align}
From the explicit form 
of $P(z)$, the residue at $z=\zeta$ is calculated directly as 
$$
\mathop{\mathrm{Res}}_{z=\zeta}P(z)\varpi(z)
=
\prod_{i=1}^{n}\frac{\theta(t, t^{2n-i-1}a_{2}a_{3}a_{4}a_{5};q)}
{(q;q)_{\infty}^{2}\theta(t^{i};q)\prod_{l=3}^{5}\theta(t^{i-1}a_{2}a_{l};q)}.
$$
On the other hand, the sum $\sum_{\nu\in \Lambda}M(\zeta q^{\nu})$ is known as van Diejen's truncated Macdonald-type sum 
\cite[p.490 Theorem 1 (2.11), (2.13)]{vD97}, 
and is evaluated as (see also \cite[p.138]{Ito2006aw} or \cite[p.641 Theorem 6.1]{Ito2006bc})
$$
\sum_{\nu\in \Lambda}M(\zeta q^{\nu})=
\prod_{i=1}^{n}\frac{(q,qt^{-i};q)_{\infty}\prod_{2\le j<k\le 5}(qa_{j}^{-1}a_{k}^{-1}t^{-(i-1)};q)_{\infty}}
{(qt^{-1},qa_{2}^{-1}a_{3}^{-1}a_{4}^{-1}a_{5}^{-1}t^{-(2n-i-1)};q)_{\infty}}.
$$
Therefore, from \eqref{eq:ResidueMacSum}, the left-hand side of \eqref{eq:KBC} is evaluated as  
$$
\prod_{i=1}^{n}\frac{\theta(t, t^{2n-i-1}a_{2}a_{3}a_{4}a_{5};q)}
{(q;q)_{\infty}^{2}\theta(t^{i};q)\prod_{l=3}^{5}\theta(t^{i-1}a_{2}a_{l};q)}
\frac{(q,qt^{-i};q)_{\infty}\prod_{2\le j<k\le 5}(qa_{j}^{-1}a_{k}^{-1}t^{-(i-1)};q)_{\infty}}
{(qt^{-1},qa_{2}^{-1}a_{3}^{-1}a_{4}^{-1}a_{5}^{-1}t^{-(2n-i-1)};q)_{\infty}},
$$
which is equal to the right-hand side of \eqref{eq:KBC}. This completes the proof.
\qed
\par
\medskip
Our remaining task is to establish
the two-term recurrence relation \eqref{eq:recurrenceJ} in Lemma \ref{lem:recurrenceJ}. As previously remarked, for its derivation we will follow Aomoto's method, which in turn requires the use of particular interpolation polynomials $\{ E_r(z) \}$ known as fundamental invariants of type $BC$.  
The latter are introduced in the next section, Section \ref{section:3}. In
Section \ref{section:4} Aomoto's method is introduced and used in combination with $\{ E_r(z) \}$.
Let $\langle \phi (z) \rangle$ denote the integral on the 
left-hand side of \eqref{eq:KBC}, with a further factor of
$\phi (z)$ in the integrand, 
and 
under the condition $q=a_{1}a_{2}a_{3}a_{4}a_{5}a_{6}t^{2n-2}$.
Aomoto's method is used to establish a recurrence between
$\langle E_r (z) \rangle$ and $\langle E_{r-1} (z) \rangle$. 
It is then shown in Section \ref{section:5} how this can be used to establish \eqref{eq:recurrenceJ}.
\section{Fundamental $BC_n$ invariants}
\label{section:3}
We denote by $W=(\mathbb{Z}/2\mathbb{Z})^{n}\rtimes\mathfrak{S}_{n}$ the Weyl group of type $C_{n}$. 
This group acts on the space $\mathcal{O}((\mathbb{C}^{*})^{n})$ of holomorphic functions on $(\mathbb{C}^{*})^{n}$ 
through permutations and inversions of the variables $z_{1},\ldots, z_{n}$.
Let ${\sf H}_{s,n}$ be the $\mathbb{C}$-linear space of $\mathbb{C}[z_1^{\pm1},\ldots,z_n^{\pm1}]$ consisting of all 
$W$-invariant Laurent polynomials $f(z)$ such that $\deg_{z_i} f(z)\le s$ for $i=1,\ldots,n$, i.e., 
$$
{\sf H}_{s,n}=\{f(z)\in\mathbb{C}[z_1^{\pm1},\ldots,z_n^{\pm1}]^{W}\,|\, \deg_{z_i} f(z)\le s\ (i=1,\ldots,n)\}.
$$
The dimension of ${\sf H}_{s,n}$ as a $\mathbb{C}$-linear space is known as $n+s\choose s$. 
\par
We use the symbol ${\sf e}(u,v)$ specified by 
$${\sf e}(u,v):=u^{-1}(1-uv)(1-uv^{-1})=u+u^{-1}-v-v^{-1},$$
which satisfies 
$$
{\sf e}(v,u)=-{\sf e}(u,v),\quad {\sf e}(u^{-1},v)={\sf e}(u,v), \quad {\sf e}(u,u)=0.
$$
We introduce a $\mathbb{C}$-basis of ${\sf H}_{1,n}$ using the functions (Laurent polynomials)
\begin{equation}
\label{eq:E_r(z)}
E_{r}(z):=\sum_{\substack{1\le i_{1}<\cdots <i_{r}\le n \\ 1\le j_{1}<\cdots <j_{n-r}\le n}}
\prod_{k=1}^{r}\frac{{\sf e}(z_{k}, a_{2}t^{i_{k}-k})}{{\sf e}(a_{1}t^{k-1},a_{2}t^{i_{k}-k})}
\prod_{l=1}^{n-r}\frac{{\sf e}(z_{l}, a_{1}t^{j_{l}-l})}{{\sf e}(a_{2}t^{l-1},a_{1}t^{j_{l}-l})},
\end{equation}
where the summation is taken over all pairs of sequences $1\le i_{1}<\cdots <i_{r}\le n$ and
$1\le j_{1}<\cdots <j_{n-r}\le n$ such that $\{i_{1},\ldots ,i_{r}\}\cup \{j_{1},\ldots ,j_{n-r}\}
=\{1,2,\ldots,n\}$. 
In particular, we have 
\begin{equation}
\label{eq:E_0(z)E_n(z)}
\begin{split}
E_{0}(z)&=\prod_{l=1}^{n}\frac{{\sf e}(z_{l},a_{1})}{{\sf e}(a_{2}t^{l-1},a_{1})}
=\prod_{i=1}^{n}\frac{(1-a_{1}z_{i})(1-a_{1}z_{i}^{-1})}
{(1-a_{1}a_{2}t^{i-1})(1-a_{1}a_{2}^{-1}t^{-(i-1)})},\\
E_{n}(z)&=\prod_{k=1}^{n}\frac{{\sf e}(z_{k},a_{2})}{{\sf e}(a_{1}t^{k-1},a_{2})}
=\prod_{i=1}^{n}\frac{(1-a_{2}z_{i})(1-a_{2}z_{i}^{-1})}
{(1-a_{1}a_{2}t^{i-1})(1-a_{2}a_{1}^{-1}t^{-(i-1)})}. 
\end{split}
\end{equation}
The functions $E_{i}(z)$ are called the {\it fundamental $BC_n$ invariants} in \cite{IN2017}. \\[10pt]
{\bf Remark.} In \cite[p.38 Definition 3.2]{IN2017}, the unit ${\sf e}(u,v)$ is defined as $u^{-1}\theta(uv,uv^{-1};p)$ instead of $u^{-1}(1-uv)(1-uv^{-1})$. 
Since we see $u^{-1}\theta(uv,uv^{-1};p)\to u^{-1}(1-uv)(1-uv^{-1})$ if $p\to 0$, 
the properties of $\{E_{i}(z)\}$ shown in \cite{IN2017} hold true here as well. 
For example, $E_{i}(z)$ does not appear to be $W$-symmetric with respect to $z$ from the definition \eqref{eq:E_r(z)}, 
but \cite{IN2017} shows that it is in fact $W$-symmetric. 
\par\medskip
\begin{prop}
\label{prop:E(zeta)=delta}
The set $\{E_i(z)\,|\,i=0,1,\ldots,n\}$ forms a $\mathbb{C}$-basis of ${\sf H}_{1,n}$ satisfying 
\begin{equation}
\label{eq:E(zeta)=delta}
E_i(\zeta_j)=\delta_{ij}\quad \mbox{for}\quad i,j \in \{0,1,\ldots,n\},
\end{equation}
where $\zeta_{j}\in (\mathbb{C}^*)^n$ $(j=0,1,\ldots,n)$ are the points specified by \eqref{eq:zeta_j} and 
$\delta_{ij}$ is the Kronecker delta. 
\end{prop} 
\proof 
See \cite[p.37 Theorem 3.1]{IN2017}. \qed
%
%
\section{Definition of the integral, and the operator $\nabla_{q,z_i}$}
\label{section:4}
The aim of this section is to derive a recurrence relation among expectation values of the fundamental 
$BC_n$ invariants, which will be used to establish the $q$-difference equation in Section \ref{section:5}. 
\par
Let $\Phi(z)$ and $\Psi(z)$ be the functions in $z=(z_1,\ldots,z_n)\in (\mathbb{C}^*)^n$ defined by 
\begin{equation}
\label{eq:BC-Phi}
\begin{split}
\Phi(z)&:=\prod_{i=1}^n\frac{(z_i^2,z_i^{-2};q)_\infty}
{\prod_{m=1}^6(a_mz_i,a_mz_i^{-1};q)}
\prod_{1\le j<k\le n}\frac{(z_jz_k^{-1},z_j^{-1}z_k,z_jz_k,z_j^{-1}z_k^{-1};q)_\infty}
{(tz_jz_k^{-1},tz_j^{-1}z_k,tz_jz_k,tz_j^{-1}z_k^{-1};q)_\infty}
\end{split}
\end{equation}
and 
\begin{equation}
\label{eq:BC-Psi}
\Psi(z)=h(z)\Phi(z),
\end{equation}
respectively, where $h(z)=\prod_{i=1}^{n}\theta(a_{6}z_{i},a_6z_i^{-1};q)$. 
Using \eqref{eq:00quasi-period}, 
the function 
$h(z)$ satisfies the quasi-periodicity 
\begin{equation}
\label{eq:Th(z)=h(z)***}
T_{q,z_i}h(z)=h(z)/qz_{i}^{2}
\quad (i=1,2,\ldots,n)
, 
\end{equation}
where $T_{q,z_i}$ stands for the $q$-shift operator in $z_i$,
$$
T_{q,z_i}f(z_1,\ldots,z_i,\ldots,z_n)=f(z_1,\ldots,qz_i,\ldots,z_n).
$$
We remark that 
the function $\Psi(z)$ coincides with the integrand of 
Gustafson's multidimensional Nassrallah--Rahman integral \eqref{eq:mNRint}. 
Further introduce the function $\tilde\Psi(z)$ defined by 
\begin{equation}
\label{eq:tPsi,Psi}
\tilde\Psi(z)=\tilde{h}(z)\Psi(z),
\end{equation}
where 
$$
\tilde{h}(z)=
\prod_{i=1}^{n}
\bigg[
\frac{\theta(t^{n-1}a_{3}a_{4}a_{5}a_{6}v^{-1}z_{i}, vz_{i};q)}{\theta(qz_{i}^{2};q)}
\frac{\prod_{k=3}^{5}\theta(a_{k}/z_{i};q)}{\theta(a_{6}z_{i};q)}
\bigg]
\prod_{1\le j< k\le n}
\frac{\theta(qt^{-1}z_{j}z_{k};q)}{\theta(qz_{j}z_{k};q)}. 
$$
Using \eqref{eq:00quasi-period}, it is confirmed that 
the function $\tilde{h}(z)$ is $q$-periodic with respect to $z$, i.e.,
$$
T_{q,z_i}\tilde{h}(z)=\tilde{h}(z)
\quad (i=1,2,\ldots,n).
$$
If we impose $q=a_{1}a_{2}a_{3}a_{4}a_{5}a_{6}t^{2n-2}$, then we see 
$\theta(t^{n-1}a_{3}a_{4}a_{5}a_{6}v^{-1}z_{i};q)=\theta(t^{n-1}a_{1}a_{2}vz_{i}^{-1};q)$.
This implies 
that the function $\tilde\Psi(z)$ coincides with the integrand of the left-hand side of \eqref{eq:KBC} 
under the condition $q=a_{1}a_{2}a_{3}a_{4}a_{5}a_{6}t^{2n-2}$. 
Moreover, from the relation \eqref{eq:tPsi,Psi}, $\tilde\Psi(z)$ can be considered equivalent to $\Psi(z)$ up to a pseudo-constant. 
\par
Under the definitions \eqref{eq:varpi(z)} and \eqref{eq:n-torus}
of $\varpi(z)$ and $\mathbb{T}^n$, respectively, 
we use the notation
\begin{equation}
\label{eq:la-ra}
\la\varphi(z)\ra=\frac{1}{n!}\int_{\mathbb{T}^n}\varphi(z)\tilde\Psi(z)\varpi(z)
\end{equation}
for any meromorphic function $\varphi(z)$ on $(\mathbb{C}^*)^n$ such that 
$\varphi(z)\tilde\Psi(z)$ is holomorphic in a neighborhood of $\mathbb{T}^n$. 
In particular, we see that the integral 
$$
\la 1\ra=\frac{1}{n!}\int_{\mathbb{T}^n}\tilde\Psi(z)\varpi(z)
$$
coincides with the left-hand side of \eqref{eq:KBC} under the condition $q=a_{1}a_{2}a_{3}a_{4}a_{5}a_{6}t^{2n-2}$.
From the definition of $\Phi(z)$ and $\Psi(z)$ we have 
\begin{align*}
\frac{T_{q,z_i}\Phi(z)}{\Phi(z)}
&=-
\frac{1-1/q^{2}z_{i}^{2}}{qz_{i}^{2}(1-z_{i}^{2})}\prod_{k=1}^6\frac{1-a_kz_i}{1-a_k/qz_i}
\\&\times
\prod_{\substack{1\le j\le n\\[1pt] j\ne i}}
\frac{(1-tz_i/z_j)(1-tz_iz_j)(1-z_j/qz_i)(1-1/qz_iz_j)}{(1-z_i/z_j)(1-z_iz_j)(1-tz_j/qz_i)(1-t/qz_iz_j)},
\end{align*}
so that we have 
\begin{align*}
\frac{T_{q,z_i}\Psi(z)}{\Psi(z)}
=\frac{T_{q,z_i}\Phi(z)}{\Phi(z)}\frac{T_{q,z_i}h(z)}{h(z)}=\frac{1}{qz_{i}^{2}}\frac{T_{q,z_i}\Phi(z)}{\Phi(z)}
%
=-\frac{F_{+,i}(z)}{T_{q,z_{i}}F_{-,i}(z)},
\end{align*}
where
\begin{align}
F_{+,i}(z)&:=\frac{1}{z_{i}^{2}}\frac{\prod_{k=1}^6(1-a_kz_i)}{1-z_{i}^{2}}
\prod_{\substack{1\le j\le n\\[1pt] j\ne i}}
\frac{(1-tz_i/z_j)(1-tz_iz_j)}{(1-z_i/z_j)(1-z_iz_j)}
\label{eq:F+i_1}\\&
=\frac{1}{z_{i}^{3+(n-1)}}\frac{\prod_{k=1}^6(1-a_kz_i)}{z_{i}^{-1}(1-z_{i}^{2})}
\prod_{\substack{1\le j\le n\\[1pt] j\ne i}}
\frac{(1-tz_i/z_j)(1-tz_iz_j)}{z_{i}^{-1}(1-z_i/z_j)(1-z_iz_j)},
\label{eq:F+i_2}\\
F_{-,i}(z)&:=F_{+,i}(z^{-1})=z_{i}^{2}\frac{\prod_{k=1}^6(1-a_k/z_i)}{1-1/z_{i}^{2}}
\prod_{\substack{1\le j\le n\\[1pt] j\ne i}}
\frac{(1-tz_j/z_i)(1-t/z_iz_j)}{(1-z_j/z_i)(1-1/z_iz_j)}
\label{eq:F-i_1}\\
&=-\frac{1}{z_{i}^{3+(n-1)}}
\frac{\prod_{k=1}^6(a_{k}-z_{i})}{z_{i}^{-1}(1-z_{i}^{2})}
\prod_{\substack{1\le j\le n\\[1pt] j\ne i}}
\frac{(t-z_i/z_j)(t-z_iz_j)}{z_{i}^{-1}(1-z_i/z_j)(1-z_iz_j)}. 
\label{eq:F-i_2}
\end{align}
Here we remark that 
\begin{equation}
\label{eq:Ratio}
\frac{T_{q,z_i}\tilde\Psi(z)}{\tilde\Psi(z)}
=
\frac{T_{q,z_i}\Psi(z)}{\Psi(z)}
=-\frac{F_{+,i}(z)}{T_{q,z_{i}}F_{-,i}(z)}.
\end{equation}
It is very important that the above equation is common for the functions $\Psi(z)$ and $\tilde\Psi(z)$, 
even if they themselves look different. 
The equation \eqref{eq:Ratio} will be used for the proof of a key lemma of Aomoto's method. 
\par
\medskip
For an arbitrary meromorphic function $\varphi(z)$ on $(\mathbb{C}^*)^n$,
we define the function $\nabla_{q,z_i}\varphi(z)$ by 
$$
\nabla_{q,z_{i}}\varphi(z):=F_{-,i}(z)\varphi(z)+F_{+,i}(z)T_{q,z_{i}}\varphi(z)
\quad\mbox{for}\quad i=1,\ldots,n.
$$
\begin{lem}
\label{lem:nabla varphi=0}
Suppose that $|a_i|<1$ $(i=1,\ldots,6)$ and $|t|<1$. 
For an arbitrary holomorphic function $\varphi(z)$ on $(\mathbb{C}^*)^n$, we have 
\begin{equation}
\label{eq:nabla varphi=0}
\la \nabla_{q,z_i}\varphi(z)\ra=0
\quad\mbox{for}\quad i=1,\ldots,n.
\end{equation}
\end{lem}
\proof 
Without loss of generality, we prove \eqref{eq:nabla varphi=0} for $i=1$, i.e., $\la \nabla_{q,z_1}\varphi(z)\ra=0$, which is equivalent to 
\begin{equation}
\label{eq:intG1phiPsi=}
\int_{\mathbb{T}^n}
F_{-,1}(z)\varphi(z)\tilde\Psi(z)\varpi(z)
=
-\int_{\mathbb{T}^n}
F_{+,1}(z)\{T_{q,z_1}\varphi(z)\}\tilde\Psi(z)\varpi(z).
\end{equation}
Since we have $F_{+,1}(z)\tilde\Psi(z)=-T_{q,z_1}F_{-,1}(z)\,T_{q,z_1}\tilde\Psi(z)$ from \eqref{eq:Ratio}, 
the right-hand side of \eqref{eq:intG1phiPsi=} is calculated as 
\begin{equation}
\label{eq:intF1TphiPsi=}
\begin{split}
&-\int_{\mathbb{T}^n}
F_{+,1}(z)\{T_{q,z_1}\varphi(z)\}\tilde\Psi(z)\varpi(z)=
\int_{\mathbb{T}^n}T_{q,z_1}\{F_{-,1}(z)\varphi(z)\tilde\Psi(z)\}\varpi(z)\\
&\ =\frac{1}{(2\pi\sqrt{-1})^n}\int_{\mathbb{T}^{n-1}}
\Bigg(\int_{|z_1|=|q|}F_{-,1}(z)\varphi(z)\tilde\Psi(z)\frac{dz_1}{z_1}\Bigg)
\frac{dz_2}{z_2}\cdots\frac{dz_n}{z_n},
\end{split}
\end{equation}
which is due to the variable change $z_1\to qz_1$. 
By definition, the integrand is written explicitly as 
\begin{equation}
\label{eq:G1phiPsi}
\begin{split}
&F_{-,1}(z)\varphi(z)\tilde\Psi(z)\\
&\quad=
-\frac{\varphi(z)}{z_{1}^{2}}
\frac{\theta(qt^{n-1}a_{3}a_{4}a_{5}a_{6}v^{-1}z_{1}, qvz_{1};q)}{\prod_{j=1}^{2}(a_jz_{1},qa_jz_{1}^{-1};q)_{\infty}}
\\&\qquad\times
\prod_{m=3}^{6}\frac{(a_{m}^{-1}z_{1};q)_{\infty}}{(a_{m}z_{1};q)_{\infty}}
\prod_{k=2}^{n}
\frac{(z_{1}/z_{k},qz_{k}/z_{1},t^{-1}z_{1}z_{k};q)_{\infty}}
{(tz_{1}/z_{k},qtz_{k}/z_{1},tz_{1}z_{k};q)_{\infty}}
\\&\qquad
\times\prod_{i=2}^{n}
\bigg[
(1-z_{i}^{2})
\frac{\theta(t^{n-1}a_{1}a_{2}vz_{i}^{-1}, vz_{i};q)}{\prod_{j=1}^{2}(a_jz_{i},a_jz_{i}^{-1};q)_{\infty}}
\prod_{m=3}^{6}\frac{(qa_{m}^{-1}z_{i};q)_{\infty}}{(a_{m}z_{i};q)_{\infty}}
\bigg]\\
&\qquad\times\prod_{2\le j< k\le n}
(1-z_{j}z_{k})
\frac{(z_{j}/z_{k},z_{k}/z_{j};q)_{\infty}}
{(tz_{j}/z_{k},tz_{k}/z_{j};q)_{\infty}}
\frac{(qt^{-1}z_{j}z_{k};q)_{\infty}}
{(tz_{j}z_{k};q)_{\infty}}.
\end{split}
\end{equation}
Focussing attention on the denominator of \eqref{eq:G1phiPsi},
when $|a_i|<1$ $(i=1,\ldots,6)$ and $|t|<1$ are satisfied, and $z_k$ $(2\le k\le n)$ are fixed as $|z_k|=1$, 
we see the right-hand side of \eqref{eq:G1phiPsi} as a function of $z_1$ has no poles in the annulus $|q|\le |z_1|\le 1$, allowing for contour deformation. 
Hence, by Cauchy's integral theorem we have 
\begin{equation}
\label{eq:Cauchy}
\int_{|z_1|=|q|}
F_{-,i}(z)\varphi(z)\tilde\Psi(z)\frac{dz_1}{z_1}
=
\int_{|z_1|=1}
F_{-,i}(z)\varphi(z)\tilde\Psi(z)\frac{dz_1}{z_1}.
\end{equation}
From \eqref{eq:intF1TphiPsi=} and \eqref{eq:Cauchy}, we obtain \eqref{eq:intG1phiPsi=}. 
\qed
\medskip
\noindent
{\bf Remark.} Since the equation \eqref{eq:Ratio} is common for $\Psi(z)$ and $\tilde\Psi(z)$, 
Lemma \ref{lem:nabla varphi=0} is also valid for the integral of $\Psi(z)$ (Gustafson's integral case). It means that, if we define 
$\la \varphi(z)\ra:=\frac{1}{2^{n}n!}\int\varphi(z)\Psi(z)\varpi(z)$ instead of \eqref{eq:la-ra}, then \eqref{eq:nabla varphi=0} remains valid. 
\par
\medskip
Let 
$E_r(z)$ $(0\le r\le n)$ be the fundamental invariants 
specified by \eqref{eq:E_r(z)}, 
and let $\zeta_j$ $(0\le j\le n)$ be the points in $(\mathbb{C}^*)^n$ specified by \eqref{eq:zeta_j}. 
When we need to specify the number of variables, we use the notation $E^{(n)}_r(z)$ and 
$\zeta_j^{(n)}$ instead of $E_r(z)$ and $\zeta_j$, respectively.
We set the functions $h_r(z)$ $(r=1,\ldots,n)$ as 
\begin{equation}
\label{eq:hr-0}
h_r(z)=\sum_{i=1}^n\nabla_{q,z_i}E_{r-1}^{(n-1)}(z_{\,\widehat{i}\,})
=
\sum_{i=1}^n\Big(F_{-,i}(z)+F_{+,i}(z)\Big)E_{r-1}^{(n-1)}(z_{\,\widehat{i}\,}),
\end{equation}
where $z_{\,\widehat{i}\,}=(z_1,\ldots,z_{i-1},z_{i+1},\ldots,z_n)$. 
Then Lemma \ref{lem:nabla varphi=0} implies that  
\begin{equation}
\label{eq:la hr ra=0}
\la h_r(z)\ra=\sum_{i=1}^n\la\nabla_{q,z_i}E_{r-1}^{(n-1)}(z_{\,\widehat{i}\,})\ra=0.
\end{equation}
\begin{lem}
\label{lem:hr}
Suppose that $a_1a_2a_3a_{4}a_{5}a_{6}t^{2n-2}=1$. 
\begin{itemize}
\item [{\rm (1)}] For each $r=1,\ldots,n$, $h_r(z)$ belongs to 
${\sf H}_{1,n}$. 
\item [{\rm (2)}]
The polynomial $h_r(z)$ is expanded in terms of $E_i(z)$ as 
$$
h_r(z)=c_{r,r}E_r(z)+ c_{r,r-1}E_{r-1}(z),
$$
where the coefficients $c_{r,r}$ and $c_{r,r-1}$ are given by 
$$
c_{r,r}=F_{+,r}(\zeta_r),\qquad c_{r,r-1}=F_{+,n}(\zeta_{r-1}).
$$
\end{itemize}
\end{lem}
\proof 
(1) First we show $h_r(z)$ is 
$W$-symmetric Laurent polynomial. In order to confirm this we slightly deform 
the expression \eqref{eq:hr-0} of $h_r(z)$. 
Since $F_{+,i}(z)$ and $F_{-,i}(z)$ given in \eqref{eq:F+i_2} and \eqref{eq:F-i_2}
are written as
\begin{equation*}
F_{+,i}(z)=\bar{F}_{+,i}(z)\frac{\Delta^{(n-1)}(z_{\,\widehat{i}\,})}{\Delta^{(n)}(z)},
\quad 
F_{-,i}(z)=-\bar{F}_{-,i}(z)\frac{\Delta^{(n-1)}(z_{\,\widehat{i}\,})}{\Delta^{(n)}(z)},
\end{equation*}
where $\bar{F}_{+,i}(z)$ and $\bar{F}_{-,i}(z)$ are specified as 
\begin{align*}
\bar{F}_{+,i}(z)&=\frac{\prod_{k=1}^6(1-a_kz_i)}{z_{i}^{3}}
\prod_{\substack{1\le j\le n\\[1pt] j\ne i}}
\frac{(1-tz_i/z_j)(1-tz_iz_j)}{z_{i}},
\\
\bar{F}_{-,i}(z)&=\frac{\prod_{k=1}^6(a_{k}-z_i)}{z_{i}^{3}}
\prod_{\substack{1\le j\le n\\[1pt] j\ne i}}
\frac{(t-z_i/z_j)(t-z_iz_j)}{z_{i}}
\end{align*}
and $\Delta^{(n)}(z)$ is the Weyl denominator of type $C_{n}$ defined as 
$$
\Delta^{(n)}(z)=\prod_{i=1}^{n}\frac{1-z_{i}^{2}}{z_{i}}\prod_{1\le j<k\le n}\frac{(1-z_{j}/z_{k})(1-z_{j}z_{k})}{z_{j}}, 
$$
we have
\begin{align}
&h_r(z)=\sum_{i=1}^n\Big(F_{-,i}(z)+F_{+,i}(z)\Big)E_{r-1}^{(n-1)}(z_{\,\widehat{i}\,})\nonumber\\
&\quad=\frac{1}{\Delta^{(n)}(z)}
\sum_{i=1}^n(-1)^{i-1}\Big(\bar{F}_{+,i}(z)-\bar{F}_{-,i}(z)\Big)E_{r-1}^{(n-1)}(z_{\,\widehat{i}}\,)\Delta^{(n-1)}(z_{\,\widehat{i}\,}).
\label{eq:hr-1}
\end{align}
Since the polynomial 
$
\sum_{i=1}^n(-1)^{i}\big(\bar{F}_{+,i}(z)-\bar{F}_{-,i}(z)\big)\Delta^{(n-1)}(z_{\,\widehat{i}\,})E_{r-1}^{(n-1)}(z_{\,\widehat{i}}\,)
$
is $W$-skew-symmetric, it is divisible by $\Delta^{(n)}(z)$, 
so that \eqref{eq:hr-1} itself come to a $W$-symmetric polynomial. 

Next we confirm that $\deg_{z_i}h_r(z)\le 1$ $(i=1,\ldots,n)$. 
The expression \eqref{eq:hr-1} of $h_r(z)$ is rewritten as 
\begin{equation}
\label{eq:hr-2}
h_r(z)=\frac{1}{2^{n}(n-1)!\Delta^{(n)}(z)}\,
\mathcal{A}\Big[\Big(\bar{F}_{+,1}(z)-\bar{F}_{-,1}(z)\Big)E_{r-1}^{(n-1)}(z_{\,\widehat{1}}\,)\Delta^{(n-1)}(z_{\,\widehat{1}\,})\Big],
\end{equation}
where the symbol $\mathcal{A}$ means the $W$-skew-symmetrization
defined by 
%
$
\mathcal{A}f(z)=\sum_{w\in W}({\rm sgn}\,w)\, w.f(z).
$
In the expansion of 
\begin{align*}
&\Big(\bar{F}_{+,1}(z)-\bar{F}_{-,1}(z)\Big)E_{r-1}^{(n-1)}(z_{\,\widehat{1}}\,)\Delta^{(n-1)}(z_{\,\widehat{1}\,})\\
&\quad =
z_1^{-3-(n-1)}
\bigg[(1-a_{1}z_1)\cdots(1-a_{6}z_1)\prod_{j=2}^n
(1-tz_1/z_{j})(1-tz_1z_{j})\\
&\qquad
-(a_{1}-z_1)\cdots(a_{6}-z_1)\prod_{j=2}^n
(t-z_1/z_{j})(t-z_1z_{j})\bigg]
E_{r-1}^{(n-1)}(z_{\,\widehat{1}}\,)\Delta^{(n-1)}(z_{\,\widehat{1}\,}),
\end{align*}
the highest term $(a_1\cdots a_6t^{2n-2}-1)z_1^{n+2}\times z_2\cdots z_n\times z_2^{n-1}z_3^{n-2}\cdots z_{n}$ vanishes under the assumption $a_1a_2a_3a_4a_5a_6t^{2n-2}=1$
(and of course, the lowest term also vanishes in the same manner). It follows that  
$$\mathcal{A}(z_1^{n+1}\times z_2\cdots z_n\times z_2^{n-1}z_3^{n-2}\cdots z_{n})
=\mathcal{A}(z_1z_2\cdots z_n\times z_1^{n}z_2^{n-1}\cdots z_{n})
$$
is now the term of highest degree up to constant in the numerator of \eqref{eq:hr-2}. 
Therefore, from the expression \eqref{eq:hr-2}, the $W$-symmetric Laurent polynomial 
$$\frac{\mathcal{A}(z_1z_2\cdots z_n\times z_1^{n}z_2^{n-1}\cdots z_{n})}{\Delta^{(n)}(z)}\in {\sf H}_{1,n}$$
is of highest degree in $h_r(z)$, so that $h_r(z)\in {\sf H}_{1,n}$.\\
(2)
Since $\{E_{i}(z)\}$ is a $\mathbb{C}$-basis of ${\sf H}_{1,n}$ and $h_r(z)\in {\sf H}_{1,n}$, $h_r(z)$ is expanded as
$
h_r(z)=\sum_{j=0}^nc_{r,j}E_j(z)
$. 
Using $E_i(\zeta_j)=\delta_{ij}$ in Proposition \ref{prop:E(zeta)=delta}, we have 
$$c_{r,j}=h_r(\zeta_j)
=\sum_{i=1}^n\Big(F_{-,i}(\zeta_j)+F_{+,i}(\zeta_j)\Big)E_{r-1}^{(n-1)}(z_{\,\widehat{i}\,})\Big|_{z=\zeta_j}.$$
By the definitions \eqref{eq:F+i_1} and \eqref{eq:F-i_1}, 
$F_{+,i}(z)$ and $F_{-,i}(z)$ satisfy the vanishing property
$$
F_{-,i}(\zeta_j)=0 \quad\mbox{if}\quad 0\le j\le n,\qquad F_{+,i}(\zeta_j)=0 \quad\mbox{if}\quad i\ne j,n \ \mbox{ and }\ 0\le j\le n. 
$$
Therefore, we have that $c_{r,j}$ for $1\le j\le n-1$ can be written as 
\begin{align*}
&c_{r,j}=h_r(\zeta_j)=
F_{+,j}(\zeta_j)E_{r-1}^{(n-1)}(z_{\,\widehat{j}\,})\Big|_{z=\zeta_j}
+F_{+,n}(\zeta_j)E_{r-1}^{(n-1)}(z_{\,\widehat{n}\,})\Big|_{z=\zeta_j}
\\[5pt]&
\ =F_{+,j}(\zeta_j)E_{r-1}^{(n-1)}(\zeta_{j-1}^{(n-1)})
+F_{+,n}(\zeta_j)E_{r-1}^{(n-1)}(\zeta_{j}^{(n-1)})
=\delta_{r,j}F_{+,j}(\zeta_j)%
+\delta_{r-1,j}F_{+,n}(\zeta_j).%
\end{align*}
In the same way 
we also have $c_{r,j}$ for $j=0$ and $j=n$ can be written 
\begin{align*}
c_{r,0}&=h_r(\zeta_0)=F_{+,n}(\zeta_0)E_{r-1}^{(n-1)}(z_{\,\widehat{n}\,})\Big|_{z=\zeta_0}
=F_{+,n}(\zeta_0)E_{r-1}^{(n-1)}(\zeta_{0}^{(n-1)})=\delta_{r,1}F_{+,n}(\zeta_0),\\
c_{r,n}&=h_r(\zeta_n)=F_{+,n}(\zeta_n)E_{r-1}^{(n-1)}(z_{\,\widehat{n}\,})\Big|_{z=\zeta_n}
=F_{+,n}(\zeta_n)E_{r-1}^{(n-1)}(\zeta_{n-1}^{(n-1)})=\delta_{r,n}F_{+,n}(\zeta_n).
\end{align*}
Thus $c_{r,j}$ vanish if $j\ne r, r-1$, and we obtain 
$c_{r,r}=F_{+,r}(\zeta_r)$ and $c_{r,r-1}=F_{+,n}(\zeta_{r-1})$. 
\qed
\begin{lem} 
\label{lem:Er=Er-1***}
Suppose that $a_1a_2a_3a_4a_5a_6t^{2n-2}=1$. 
For $r=1,\ldots,n$ we have
\begin{equation}
\label{eq:Er=Er-1***}
\begin{split}
\la E_{r}(z)\ra
&= -\la E_{r-1}(z)\ra 
\prod_{k=3}^6\frac{1-a_2a_kt^{n-r}}{1-a_1a_kt^{r-1}}
\\&\qquad\times\frac
{a_1^{2}t^{2r-2}(1-t^{n-r+1})(1-a_1a_2^{-1}t^{-n+2r})(1-a_{2}a_{1}^{-1}t^{n-r+1})}
{a_2^{2}t^{2n-2r}(1-t^{r})(1-a_{1}a_{2}^{-1}t^{r})(1-a_2a_1^{-1}t^{n-2r+2})}.
\end{split}
\end{equation}
\end{lem}
\proof
Since we have $\la h_r(z)\ra=0$ by \eqref{eq:la hr ra=0}, under $a_1a_2a_3a_4a_5a_6t^{2n-2}=1$, 
Lemma \ref{lem:hr} (2) implies
$$\la E_r(z)\ra=-\frac{c_{r,r-1}}{c_{r,r}}\la E_{r-1}(z)\ra=-\frac{F_{+,n}(\zeta_{r-1})}{F_{+,r}(\zeta_{r})}\la E_{r-1}(z)\ra. $$
The proportionality of \eqref{eq:Er=Er-1***} is 
then seen to be due to the direct evaluations
\begin{equation*}
\begin{split}
F_{+,n}(\zeta_{i-1})&=
\frac{(1-t^{n-i+1})(1-a_1a_2t^{n-1})(1-a_2a_1^{-1}t^{n-i+1})}{(a_2t^{n-i})^{2}(1-t)(1-a_2a_1^{-1}t^{n-2i+2})}
\prod_{k=3}^6(1-a_ka_2t^{n-i}),\\
F_{+,i}(\zeta_i)&=\frac{(1-t^{i})(1-a_1a_2t^{n-1})(1-a_{1}a_{2}^{-1}t^{_{i}})}{(a_1t^{i-1})^{2}(1-t)(1-a_1a_2^{-1}t^{-n+2i})}
\prod_{k=3}^6(1-a_ka_{1}t^{i-1}),\\
\end{split}
\end{equation*}
 for $0\le i\le n$. 
 \qed
\medskip
\noindent
\begin{cor} 
\label{cor:En=E0***}
Suppose that $a_1a_2a_3a_4a_5a_6t^{2n-2}=1$. Then we have 
\begin{equation}
\label{eq:En=E0***}
\la E_n(z)\ra=\la E_0(z)\ra
\frac{a_{1}^{3n}}{a_{2}^{3n}}\prod_{i=1}^n
\bigg[
\frac{1-a_{2}a_{1}^{-1}t^{i-1}}{1-a_{1}a_{2}^{-1}t^{i-1}}
\prod_{k=3}^6\frac{1-a_2a_kt^{i-1}}{1-a_1a_kt^{i-1}}
\bigg].
\end{equation}
\end{cor}
\proof 
Lemma \ref{lem:Er=Er-1***} implies  that
\begin{align*}
&\frac{\la E_n(z)\ra}{\la E_0(z)\ra}
=\prod_{r=1}^n \frac{\la E_{r}(z)\ra}{\la E_{r-1}(z)\ra}
=(-1)^{n}\prod_{r=1}^n \bigg\{\prod_{k=3}^6\frac{1-a_2a_kt^{n-r}}{1-a_1a_kt^{r-1}}\\
&
\qquad\qquad\qquad\times \frac{a_1^{2}t^{2r-2}(1-t^{n-r+1})(1-a_1a_2^{-1}t^{-n+2r})(1-a_{2}a_{1}^{-1}t^{n-r+1})}
{a_2^{2}t^{2n-2r}(1-t^{r})(1-a_{1}a_{2}^{-1}t^{r})(1-a_2a_1^{-1}t^{n-2r+2})}
\bigg\}.
\end{align*}
After simplification this reduces to 
\eqref{eq:En=E0***}. \qed
%
\section{$q$-Difference equation and the proof of Lemma \ref{lem:recurrenceJ}}
\label{section:5}
Based on the results of the previous section, 
in this section we prove Lemma \ref{lem:recurrenceJ}. 
For this, we will replace the position of $a_{2}$ in the fundamental $BC$ invariants $\{E_{i}(z)\}$ with that of $a_{6}$.  
In order to investigate the dependence of the integral on the parameters $(a_{1},\ldots,a_{6})$, we set 
$$\mathcal{J}(a_{1},\ldots,a_{6})=\la 1\ra=\frac{1}{n!}\int_{\mathbb{T}^n}\tilde\Psi(z)\varpi(z).$$
From the definition \eqref{eq:tPsi,Psi} of $\tilde\Psi(z)$ with the setting $v=a_{6}$,  we have 
\begin{align*}
\frac{T_{q,a_1}\tilde\Psi(z)}{\tilde\Psi(z)}
&=\prod_{i=1}^{n}(1-a_{1}z_{i})(1-a_{1}/z_{i})
=a_{1}^{n}\prod_{i=1}^{n}{\sf e}(a_{1},z_{i})
=E_0(z)\prod_{i=1}^n a_{1} {\sf e}(a_{1},a_{6}t^{i-1}),\\
\frac{T_{q,a_6}\tilde\Psi(z)}{\tilde\Psi(z)}
&=\prod_{i=1}^{n}(1-a_{6}^{-1}z_{i})(1-a_{6}^{-1}/z_{i})=a_{6}^{-n}\prod_{i=1}^{n}{\sf e}(a_{6},z_{i})
=E_n(z)\prod_{i=1}^n a_{6}^{-1} {\sf e}(a_{6},a_{1}t^{i-1}).
\end{align*}
Consequently 
$$
\mathcal{J}(a_{1},\ldots,qa_{6})=\la E_n(z) \ra \prod_{i=1}^n a_{6}^{-1} {\sf e}(a_{6},a_{1}t^{i-1}),
\quad
\mathcal{J}(qa_{1},\ldots,a_{6})=\la E_0(z) \ra \prod_{i=1}^na_{1} {\sf e}(a_{1},a_{6}t^{i-1}).
$$
By \eqref{eq:En=E0***} in Corollary \ref{cor:En=E0***}, we then obtain 
\begin{equation}
\label{eq:J/J}
\begin{split}
&\frac{\mathcal{J}(a_{1},\ldots,qa_{6})}{\mathcal{J}(qa_{1},\ldots,a_{6})}
=\frac{\la E_n(z)\ra}{\la E_0(z)\ra}
a_{1}^{-n}a_{6}^{-n}\prod_{i=1}^n\frac{{\sf e}(a_{6},a_{1}t^{i-1})}{{\sf e}(a_{1},a_{6}t^{i-1})}
\\&
\quad=\frac{a_{1}^{2n}}{a_{6}^{4n}}\prod_{i=1}^n
\bigg[
\frac{{\sf e}(a_{1}t^{i-1},a_{6})(1-a_{6}a_{1}^{-1}t^{i-1})}{{\sf e}(a_{6}t^{i-1},a_{1})(1-a_{1}a_{6}^{-1}t^{i-1})}
\prod_{k=2}^5\frac{1-a_6a_kt^{i-1}}{1-a_1a_kt^{i-1}}
\bigg]
\\&
\quad=\frac{a_{1}^{2n}}{a_{6}^{4n}}\prod_{i=1}^n
\bigg[
\frac{a_{6}}{a_{1}}
\prod_{k=2}^5\frac{1-a_6a_kt^{i-1}}{1-a_1a_kt^{i-1}}
\bigg]
\\&
\quad=
(a_{1}a_{2}a_{3}a_{4}a_{5}a_{6}t^{2n-2})^{n}\prod_{i=1}^n
\prod_{k=2}^5\frac{1-a_6^{-1}a_k^{-1}t^{-(i-1)}}{1-a_1a_kt^{i-1}}
\\&
\quad=\prod_{i=1}^n\prod_{k=2}^5\frac{1-a_6^{-1}a_k^{-1}t^{-(i-1)}}{1-a_1a_kt^{i-1}},
\end{split}
\end{equation}
where $a_1a_2a_3a_4a_5a_6t^{2n-2}=1$. 
\par
Now replace $a_6$ with $q^{-1}a_6$. The above
balancing condition is then restored to the form 
$a_1a_2a_3a_4a_5a_6t^{2n-2}=q$ in
Theorem \ref{thm:KBC}. 
Further, we have the $q$-difference equation
$$
\frac{\mathcal{J}(a_{1},a_{2},\ldots,a_{5},a_{6})}{\mathcal{J}(qa_{1},a_{2},\ldots,a_{5},q^{-1}a_{6})}
=\prod_{i=1}^n\prod_{k=2}^5\frac{1-qa_6^{-1}a_k^{-1}t^{-(i-1)}}{1-a_1a_kt^{i-1}}.
$$
This is what we wanted to prove in Lemma \ref{lem:recurrenceJ}. 
\appendix
\section{Transformation of parameters and deformation of the cycle}
\label{appendix:A}
In this appendix, we explain how the integral \eqref{eq:KBCe} can be derived from \eqref{eq:KBC}; 
the integral \eqref{eq:KBCe-2} can  be obtained in exactly the same way. 
We fix the complex number $\varepsilon$ satisfying $|\varepsilon|=\epsilon$, where $0<\epsilon<1$. 
We first move the parameters $a_{1}$, $a_{2}$ and $t$ from the range 
$|a_{i}|<1$ to 
$|a_{i}|<\epsilon$ $(i=1,2)$ and $|t|<1$ to 
$|t|<\epsilon$, i.e., we impose that
\begin{equation}
\label{eq:condition-A}
|a_{1}|<\epsilon,\  |a_{2}|<\epsilon,\  
|a_{3}|<1,\ |a_{4}|<1,\ |a_{5}|<1,\  |a_{6}|<1
\quad
\mbox{and}\quad
|t|<\epsilon,
\end{equation}
under the balancing condition $q=a_{1}a_{2}a_{3}a_{4}a_{5}a_{6}t^{2n-2}$. 
After we reset the parameters in \eqref{eq:KBC} as 
$$a_{1}\to a_{1}\varepsilon,\ 
a_{2}\to a_{2}\varepsilon,\ 
a_{3}\to b_{1}\varepsilon^{-1},\ 
a_{4}\to b_{2}\varepsilon^{-1}, \
a_{5}\to b_{3}\varepsilon^{-1}, \
a_{6}\to b_{4}\varepsilon^{-1}\
\mbox{and}\ 
v\to v \varepsilon^{-1},$$ 
the range \eqref{eq:condition-A} is rewritten as 
$$
|a_{1}\varepsilon|<\epsilon,\  |a_{2}\varepsilon|<\epsilon,\  
|b_{1}\varepsilon^{-1}|<1,\ |b_{2}\varepsilon^{-1}|<1,\ |b_{3}\varepsilon^{-1}|<1,\  |b_{4}\varepsilon^{-1}|<1\quad
\mbox{and}\quad 
|t|<\epsilon.
$$
The condition for the parameters 
then becomes
\begin{equation}
\label{eq:condition-A2}
|a_{1}|<1,\  |a_{2}|<1,\  
|b_{1}|<\epsilon,\ |b_{2}|<\epsilon,\ |b_{3}|<\epsilon,\  |b_{4}|<\epsilon
\quad
\mbox{and}\quad 
|t|<\epsilon,
\end{equation}
under the balancing condition $q\varepsilon^{2}=a_{1}a_{2}b_{1}b_{2}b_{3}b_{4}t^{2n-2}$. 
Moreover, we apply the transformation of the integral variables as $z_{i}\to z_{i}\varepsilon$ $(i=1,\ldots, n)$
to the integral \eqref{eq:KBC}. 
Under the balancing condition 
$q\varepsilon^{2}=a_{1}a_{2}b_{1}b_{2}b_{3}b_{4}t^{2n-2}$, 
for $c\in \mathbb{R}$ satisfying $1\le c\le \epsilon^{-1}$, we
introduce the $c$-dependent integral  
\begin{equation}
\label{eq:KBCe-A}
\begin{split}
&\frac{1}{n!}\int_{\mathbb{T}_{c}^{n}}
\prod_{i=1}^{n}
\bigg[
(1-\varepsilon^{2} z_{i}^{2})
\frac{\theta(t^{n-1}a_{1}a_{2}vz_{i}^{-1}, vz_{i};q)}{\prod_{j=1}^{2}(a_j\varepsilon^{2} z_{i},a_jz_{i}^{-1};q)_{\infty}}
\prod_{k=1}^{4}\frac{(qb_{k}^{-1}\varepsilon^{2} z_{i};q)_{\infty}}{(b_{k}z_{i};q)_{\infty}}
\bigg]
\\&\qquad\quad
\times\prod_{1\le j< k\le n}
\!\!\!(1-\varepsilon^{2} z_{j}z_{k})
\frac{(z_{j}/z_{k},z_{k}/z_{j};q)_{\infty}}
{(tz_{j}/z_{k},tz_{k}/z_{j};q)_{\infty}}
\frac{(qt^{-1}\varepsilon^{2} z_{j}z_{k};q)_{\infty}}
{(t\varepsilon^{2} z_{j}z_{k};q)_{\infty}}
\,\varpi(z),
\end{split}
\end{equation}
where
\begin{equation*}
\mathbb{T}_{c}^n:=\{(z_1,\ldots,z_n)\in \mathbb{C}^n\,|\, |z_i|=c\ (i=1,\ldots,n)\}.
\end{equation*}
Then the integral \eqref{eq:KBC} under the condition \eqref{eq:condition-A} is rewritten as 
the integral \eqref{eq:KBCe-A} under the condition \eqref{eq:condition-A2} for the case $c=\epsilon^{-1}$. 
Our claim is now that the integral \eqref{eq:KBCe-A} as a function of $c$ does not change 
even if $c$ is moved continuously from $\epsilon^{-1}$ to $1$.
\par
Since the integrand of \eqref{eq:KBCe-A} is symmetric for $z_{1},\ldots, z_{n}$, 
without loss of generality we regard the integrand of \eqref{eq:KBCe-A} as a function of $z_{1}$ 
under $|z_{2}|=|z_{3}|=\cdots=|z_{n}|=c$, 
and it suffices to investigate the situation of the poles of the integrand of \eqref{eq:KBCe-A} in the $z_{1}$-plane. 
\begin{figure}[htbp]
\begin{center}
\vspace{-5pt}
\includegraphics[width=250pt]{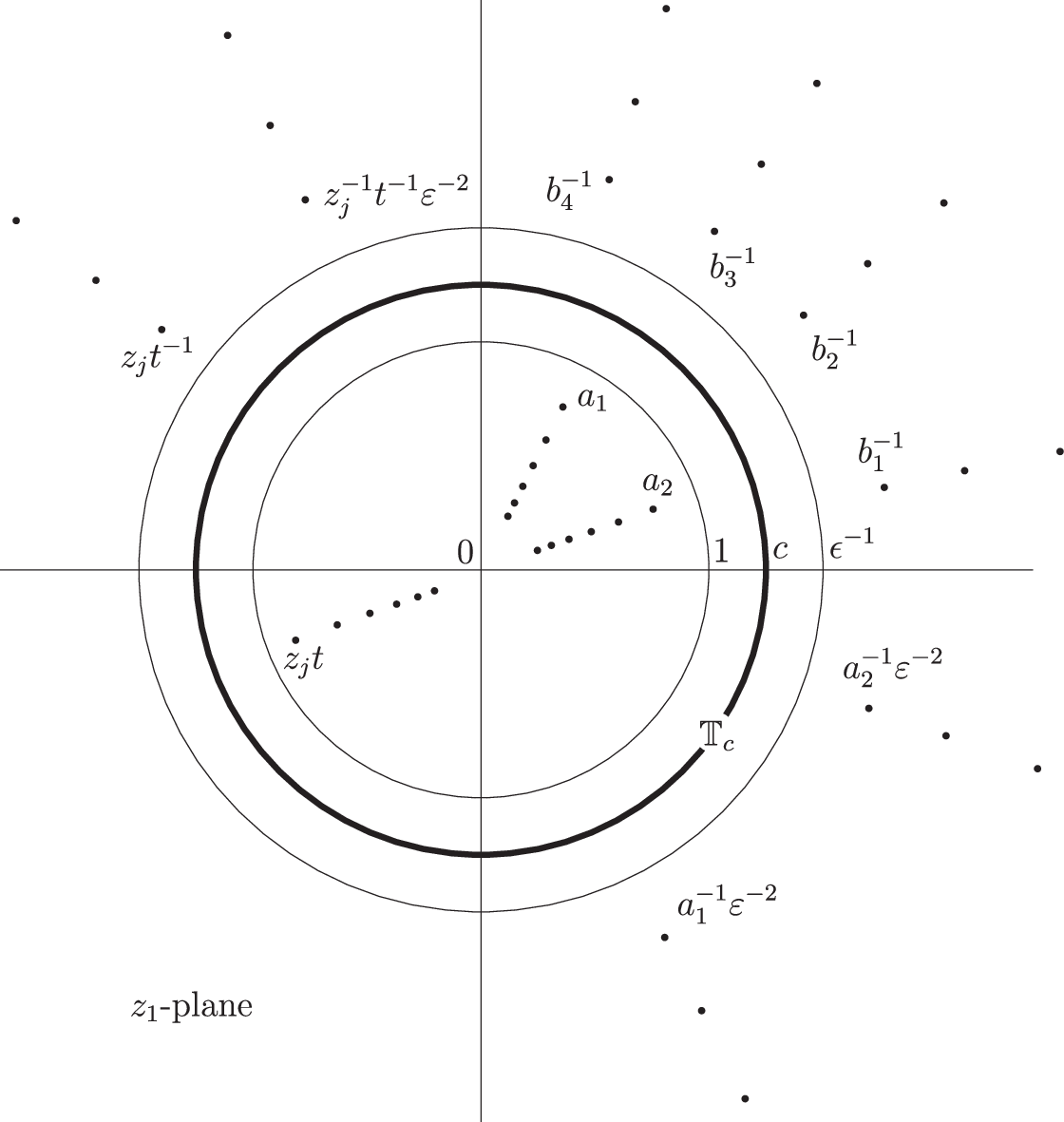}
\vspace{-5pt}
\caption{Poles in $z_{1}$-plane}
\end{center}
\vspace{-15pt}
\end{figure}
%
A listing of the poles in the $z_{1}$-plane is 
\begin{equation*}
\begin{split}
z_{1}&=a_{i}q^{\nu},\ a_{i}^{-1}\varepsilon^{-2}q^{-\nu}\ (i=1,2;\, \nu=0,1,2,\ldots),
\\
z_{1}&=b_{i}^{-1}q^{-\nu}\ (i=1,2,3,4;\, \nu=0,1,2,\ldots),
\\
z_{1}&=z_{j}tq^{\nu},\ z_{j}t^{-1}q^{-\nu},\ z_{j}^{-1}t^{-1}\varepsilon^{-2}q^{-\nu}\ 
(j=2,\ldots,n;\, \nu=0,1,2,\ldots),
\end{split}
\end{equation*}
where $|z_{2}|=|z_{3}|=\cdots=|z_{n}|=c$. Since we can confirm that 
\begin{equation*}
\begin{split}
&|a_{i}q^{\nu}|<1,\ |a_{i}^{-1}\varepsilon^{-2}q^{-\nu}|>\epsilon^{-2}\ (i=1,2;\, \nu=0,1,2,\ldots),
\\
&|b_{i}^{-1}q^{-\nu}|>\epsilon^{-1}\ (i=1,2,3,4;\, \nu=0,1,2,\ldots),
\\
&
|z_{j}tq^{\nu}|<c\epsilon<1,\ |z_{j}t^{-1}q^{-\nu}|>c\epsilon^{-1}>\epsilon^{-1},\ 
|z_{j}^{-1}t^{-1}\varepsilon^{-2}q^{-\nu}|>c^{-1}\epsilon^{-3}> \epsilon^{-2}\\ 
&\hspace{220pt}(j=2,\ldots,n;\, \nu=0,1,2,\ldots)\\[-5pt]
\end{split}
\end{equation*}
under the condition \eqref{eq:condition-A2}, the poles 
$$
z_{1}=a_{1}q^{\nu},\ a_{2}q^{\nu},\ z_{j}tq^{\nu}\ (j=2,\ldots,n;\,\nu=0,1,2,\ldots)
$$
are located inside the circle $|z_{1}|=1$, and 
the other poles 
are located outside the circle $|z_{1}|=\epsilon^{-1}$ (see Figure 1). 
Since 
the integrand of \eqref{eq:KBCe-A} 
as a function of $z_{1}$ has no poles in the annulus $1\le |z_{1}|\le \epsilon^{-1}$, 
by Cauchy's integral theorem the integral \eqref{eq:KBCe-A} under the condition \eqref{eq:condition-A2} for $c=\epsilon^{-1}$ coincides with that for $c=1$. 
Once we obtain the integral \eqref{eq:KBCe-A} under the condition \eqref{eq:condition-A2} for $c=1$, 
we can move 
the parameters $b_{j}$ $(j=1,\ldots,4)$ and $t$ 
from the range $|b_{j}| < \epsilon$ to $|b_{j}| <1$ and $|t| < \epsilon$ to $|t| <1$. 
Therefore we eventually obtain the formula \eqref{eq:KBCe} for the parameters 
satisfying  
$|a_i|<1$ $(i=1,2)$, $|b_j|<1$ $(j=1,\ldots,4)$ and $|t|<1$, 
under the balancing condition $q\varepsilon^{2}=a_{1}a_{2}b_{1}b_{2}b_{3}b_{4}t^{2n-2}$.

\section*{Acknowledgements} 
This work is supported by 
the Australian Research Council Discovery Project grant DP250102552 and 
JSPS Kakenhi Grant Number  (C)23K03153. 

\bibliographystyle{amsplain}

{\footnotesize

}

\end{document}